\documentclass{amsart}
\usepackage[utf8]{inputenc}
\usepackage{graphicx} 
\usepackage{epsfig, amsmath, amssymb, amsthm, times}
\usepackage{pdfpages}

\usepackage[utf8]{inputenc}

\usepackage{mathtools}
\mathtoolsset{showonlyrefs}

\usepackage{amsmath,amssymb,amsfonts,amsthm,bbm}
\usepackage{commath}
\usepackage{mathabx}
\usepackage{graphicx}
\usepackage{subfigure}
\usepackage{cite}
\usepackage{hyperref}
\usepackage{url}
\usepackage{mathrsfs}  
\usepackage{verbatim}

\usepackage[T1]{fontenc}
\usepackage{esint}
\usepackage{xcolor}

\numberwithin{equation}{section}

\theoremstyle{defintion}

\theoremstyle{assumption}

\theoremstyle{remark}

\theoremstyle{remark}

\theoremstyle{example}

\newtheorem{theorem}{Theorem}[section]

\newtheorem{lemma}[theorem]{Lemma}
\newtheorem{remark}[theorem]{Remark}
\newtheorem{definition}[theorem]{Definition}

\newtheorem{main-idea}[theorem]{Main Ideas}
\newtheorem{main-result}[theorem]{Main results}
\newtheorem{cor}[theorem]{Corollary}

\DeclareMathOperator{\Var}{Var}

\title[Distributional Convergence of the Empirical Laplacians]{Distributional Convergence of the Empirical Laplacians with integral Kernels on domains with boundaries}
\author{Bernard Akwei, Luke~G. Rogers, Alexander Teplyaev }
\date{\today}

\begin{document}

\begin{abstract}
 Motivated by the problem of understanding theoretical bounds for the performance of the Belkin-Niyogi Laplacian eigencoordinate approach to dimension reduction in machine learning problems, we consider the convergence of random graph Laplacian operators to a Laplacian-type operator on a manifold.  

For $\{X_j\}$ i.i.d.\ random variables taking values in $\mathbb{R}^d$ and $K$ a kernel with suitable integrability we define random graph Laplacians
\begin{equation*}
D_{\epsilon,n}f(p)=\frac{1}{n\epsilon^{d+2}}\sum_{j=1}^nK\left(\frac{p-X_j}{\epsilon}\right)(f(X_j)-f(p))
\end{equation*}
and study their convergence as $\epsilon=\epsilon_n\to0$ and $n\to\infty$ to a second order elliptic operator of the form 
\begin{align*}
\Delta_K f(p) &=  \sum_{i,j=1}^d\frac{\partial f}{\partial x_i}(p)\frac{\partial g}{\partial x_j}(p)\int_{\mathbb{R}^d}K(-t)t_it_jd\lambda(t)\\
&\quad +\frac{g(p)}{2}\sum_{i,j=1}^d\frac{\partial^2f}{\partial x_i\partial x_j}(p)\int_{\mathbb{R}^d}K(-t)t_it_jd\lambda(t).
\end{align*}

Our results provide conditions that guarantee that  $D_{\epsilon_n,n}f(p)-\Delta_Kf(p)$ converges to zero in probability as $n\to\infty$ and can be rescaled by $\sqrt{n\epsilon_n^{d+2}}$ to satisfy a central limit theorem.  They generalize the work of  Gin\'e--Koltchinskii~\cite{gine2006empirical} and Belkin--Niyogi~\cite{belkin2008towards} to allow manifolds with boundary and a wider choice of kernels $K$, and to prove convergence under weaker smoothness assumptions  and a correspondingly more precise choice of conditions on the asymptotics of $\epsilon_n$ as $n\to\infty$. 
\end{abstract}

\maketitle
\tableofcontents

\section{Introduction}

The Laplacian eigenmap method introduced by Belkin--Niyogi~\cite{belkin2003laplacian} for dimension reduction in machine learning problems may be thought of as follows: given sample points $X_1,\dotsc,X_n$ in a high dimensional space, one forms a graph Laplacian by introducing edge weights between pairs $X_i,X_j$.  Computing the eigenfunctions $\{\zeta_k\}_{k=1}^N$ of this Laplacian that have the smallest non-zero eigenvalues provides a map $(\zeta_1,\dotsc,\zeta_N):\{X_j\}\to\mathbb{R}^N$ called the Laplacian eigenmap to $\mathbb{R}^N$.  It was proposed in~\cite{belkin2003laplacian} that this map preserves the large-scale geometry of the sample data while reducing the dimension to a prescribed value $N$ because the graph Laplacians  approximate the Laplace-Beltrami operator of the manifold and the latter encodes suitable geometric information.
If the data lies on a $d$-dimensional manifold then, following \cite{belkin2008towards}, one usually has $N > d$ because the number of eigenfunctions 
needs to be larger than the dimension of the manifold for an accurate representation of this manifold. 
For instance, one needs at least three eigenfunctions to construct a three-dimensional 
representation of a two-dimensional torus, as shown in~\cite[Figure 2]{eigenmaps}.
Our motivation is to explore eigenmaps in more general, less smooth settings, offering a broader perspective compared to previous work.

The preceding sketch of the eigenmap method leaves several significant details open.  
One is how to choose the edge weights. Belkin--Niyogi suggest several possibilities,  
including connecting points $x_i, x_j$ with an edge of weight $e^{-|x_i-x_j|^2/\epsilon}$ (Gaussian weights)  
or choosing a threshold $\epsilon$ and connecting points that satisfy $|x_i-x_j| < \epsilon$.  
The highest generality would allow weights $K(x,y)$, but it is typical to only consider functions  
$K(x-y)$ or, in order to fix a length scale $\epsilon$, $K\bigl(\frac{x-y}{\epsilon}\bigr)$  
for some integrable $K$.

Another natural question is in what sense and under what assumptions the graph Laplacians approximate a Laplacian operator on the manifold.  One approach to this problem is to take points $\{X_j\}_{j=1}^n$ randomly on a manifold and ask whether the graph Laplacians converge to the classical Laplace-Beltrami operator as the length scale $\epsilon\to0$ and the number of points $n\to\infty$.    Significant early results of this type are   due to Gin\'{e}--Koltchinskii~\cite{gine2006empirical}, who write their random graph Laplacian in the form
\begin{equation}\label{eqn:defrandomgraphLapl}
D_{\epsilon,n}f(p)= \frac{1}{n\epsilon^{d+2}}\sum_{j=1}^nK\Bigl(\frac{p-X_j}{\epsilon_n}\Bigr)(f(X_j)-f(p)),
\end{equation}
with $K$ the Gaussian kernel and the points $X_i$ taken to be i.i.d.\ with uniform distribution on a compact manifold of dimension $d$ without boundary. Normalizing the measure to be unity they denote the Laplace-Beltrami operator by $\Delta$ and for a function $f$ that is $C^3$ in a neighborhood of $p$ they prove a law of large numbers (LLN) \cite[Proposition 4.1]{gine2006empirical} and a central limit theorem (CLT) \cite[Proposition 4.1]{gine2006empirical} as follows:
\begin{align}
& \text{\mbox{\qquad If }} \epsilon_n\to0 \text{ and } n\epsilon_n^{d+2}\to\infty \text{ then }   D_{\epsilon_n,n} f(p)-\Delta f(p)   \overset{\text{prob}}\longrightarrow 0, \label{eq-WLLN} \\
& \text{\mbox{\qquad If }} n\epsilon_n^{d+4}\to0 \text{ and } n\epsilon_n^d\to\infty 
 \text{ then } \sqrt{n\epsilon_n^{d+2}}\bigl( D_{\epsilon_n,n}f(p) - \Delta f(p)\bigr)  \overset{\text{dist}}\longrightarrow sZ. \label{eqn:gkcondits}
\end{align}
where $Z$ is a standard normal random variable. 
A formula for the variance $s^2$ similar to~\eqref{eqn:varfactor} below is also given in~\cite[Proposition 4.1]{gine2006empirical}.  This convergence may be taken to be uniform over a family of functions with uniformly bounded $C^3$ norm with the same proof.  Moreover, using a more difficult analysis and with extra factors of the form $|\log \epsilon_n|^{-1}$ in the hypotheses, the above convergences may be taken to be uniform over points $p$ in the manifold and the weak law of large numbers may be upgraded to a strong law (i.e.\ with almost sure convergence).

A contemporaneous LLN result of Belkin--Niyogi~\cite{belkin2008towards} is formulated  for $f\in C^\infty$ and  with $\epsilon_n=n^{-\frac1{d+2+\nu}}$ for some $\nu>0$ and may be stated in the form
\begin{equation}\label{eqn:BNlln}
  D_{\epsilon_n,n} f(p)-\Delta f(p)   \overset{\text{prob}}\longrightarrow 0 \text{ as } n\to\infty.
\end{equation}
It, too, may be taken to be uniform for $f$ in a family of functions with uniformly bounded $C^3$ norm. 
Belkin--Niyogi in \cite[Section 5]{belkin2008towards} also consider the question of taking a non-uniform probability distribution on the $X_i$ and describe the limit in this case.

Other significant results at around the same time as the works~\cite{gine2006empirical,belkin2008towards,singer2006graph} are that of Coifman--Lafon~\cite{coifman2006diffusion}, who point out that various kernel eigenmap techniques are equivalent to  (heat) diffusion maps and study of the effect of non-uniform probability distributions for sample points in this setting, and Jones--Maggioni--Schul~\cite{jones2008manifold}, who prove that a diffusion map yields a bi-Lipschitz parametrization of a H\"older manifold on a (quantitatively) small neighborhood of a point.

During the nearly 20 years since the preceding papers were written, the question of convergence of graph Laplacians to Laplace-Beltrami operators has been explored by several authors in a number of different contexts, particularly with respect to convergence of spectral data. We note, in particular, the works~\cite{burago2015graph,burago2019,garcia2020error} which study operators like~\eqref{eqn:defrandomgraphLapl} and~\eqref{avkernop} in non-probabilistic settings. Other applications of graph Laplacians like~\eqref{eqn:defrandomgraphLapl} have also been discovered, such as its use in~\cite{green2021minimax} for non parametric regression.  The effectiveness of Laplacian eigenmaps for preserving geometry and representing data in applications, along with various explanations for this effectiveness, have been considered  in~\cite{lin2022varadhan,Venkatraman2023,armstrong2025optimal}.

Our main purpose in the present work is to significantly generalize and improve upon the results of Gin\'{e}--Koltchinskii~\cite{gine2006empirical} and Belkin--Niyogi~\cite{belkin2008towards} with regard to the following aspects of the theory:
\begin{itemize}
\item Permitting more general kernels $K$ that satisfy integral estimates, including those that depend on two parameters and discontinuous kernels that arise naturally in clustering and non-parametric methods;
\item Relaxing the smoothness requirements from $f\in C^3$ to $f\in C^{2,\theta}$, i.e.\ with H\"older continuous second derivatives;
\item Allowing non-uniform distribution with density $g$ for the sample points $X_i$ and considering appropriate assumptions on the smoothness of $g$;
\item Considering   non-compact manifolds;
\item Considering  manifolds with boundary and analyzing the effect of the boundary on the limit of the graph Laplacians.
\end{itemize}
These extensions of the theory are  highly relevant for a broad range of theoretical and applied problems because data sets often have non-uniform structures and natural boundaries, including those that would make it necessary to consider functions of lower regularity.

We follow the approach common to~\cite{gine2006empirical,belkin2008towards} of introducing an operator that averages the difference $f(x)-f(p)$ as weighted by the kernel $K$ and comparing it to both $D_{\epsilon_n,n}f(p)$ and a suitable Laplacian. Define
\begin{equation}\label{avkernop}
	D_\epsilon f(p) =\frac{1}{\epsilon^{d+2}}\mathbb{E}K\left(\frac{p-X}{\epsilon}\right)(f(X)-f(p)).
	\end{equation}
	
Similarly to 	\cite[Section 5]{gine2006empirical}, we can identify the manifold with its image under the local coordinate chart because derivatives of functions are evaluated locally and our global estimates are expressed in integral terms.  Thus, consider kernels $K(p,t)$ that can depend on two variables $p,t\in\mathbb R^2$, 
although in our notation we write $K(t)$ because only dependence on $t$ is important for our results.  
	
In Section~\ref{approximation} we consider the scaled difference $\sqrt{n\epsilon_n^{d+2}}\bigl(D_{\epsilon,n}f(p)- D_{\epsilon}f(p)\bigr)$ between the averaging integral Laplacian \eqref{avkernop} and the random empirical graph Laplacian \eqref{eqn:defrandomgraphLapl} at~$p$. Our first result,  Theorem~\ref{thm2.4},  is a central limit theorem when $\epsilon_n\to0$ and $n\epsilon_n^d\to\infty$, which in particular implies that the difference $D_{\epsilon,n}f(p)- D_{\epsilon}f(p)$ converges to zero in probability exactly when $n\epsilon_n^{d+2}\to\infty$ as $n\to\infty$.
An analogous result is in the proof of Proposition~4.1 of~\cite{gine2006empirical}, but under stronger smoothness assumptions on $f$ and only for the case that $K$ is Gaussian.

 We next examine the difference $D_\epsilon f(p)-\Delta_Kf(p)$, where $\Delta_K$ is a Laplacian associated to the kernel $K$ (see Definition~\ref{delta}), showing this converges to zero if $f\in C^2$ (in Theorem~\ref{thm0.4}) and does so at rate $O(\epsilon^\theta)$ if $f\in C^{2,\theta}$ (in Theorem~\ref{thm0.5}).  These results require the probability density $g$ for the sample points to be $C^1$ or $C^{1,\theta}$ respectively.   Combining these results gives a central limit theorem when $n\epsilon_n^{d+2+2\theta}\to0$ and $n\epsilon_n^d\to\infty$ under the assumption that  $f\in C^{2,\theta}$, $g\in C^{1,\theta}$, see Corollary~\ref{combine}.  This generalizes the results of Gin\'{e}--Koltchinskii~\cite{gine2006empirical} to a lower order of smoothness and sharpens their conditions~\eqref{eqn:gkcondits} on the rate at which $\epsilon_n\to0$.

At the same time, the preceding result readily yields a weak law of large numbers.  
This is similar to the condition for the Belkin--Niyogi law of large 
numbers~\eqref{eqn:BNlln} from~\cite{belkin2008towards}. 
For the CLT,
if  $n\epsilon_n^{d+2}\to\infty$ and  $n\epsilon_n^{d+2+2\theta}\to0$,
our Corollary~\ref{combine} has an 
advantage that the exponent $\theta$ is explicitly tied to the 
smoothness of $f$ and $g$.

For both the CLT and the LLN our class of kernels and allowable distributions of sample points are more general than those considered in either of~\cite{gine2006empirical,belkin2008towards}. This does force us to consider more general Laplacians (see Definition~\ref{delta}) and to require an extra cancellation condition in the hypotheses of Theorems~\ref{thm0.4} and \ref{thm0.5}.  However, it also means that we can see understand the role symmetries of $K$ and/or the density $g$ of the distribution of the sample points could play in the convergence of $D_{\epsilon_n,n}f(p)$ to $\Delta_Kf(p)$ and the structure of $\Delta_K$.  In particular we note that in those theorems:
\begin{itemize}
	\item The cancellation assumption $\nabla f(p)\cdot\int K(-t)td\lambda=0$ is valid for even kernels, meaning those for which $K(-t)=K(t)$.
	\item  If the kernel satisfies $\int K(-t)t_it_j=0$ for $i\neq j$, then the non-diagonal terms in $\Delta_K f(p)$ vanish. This happens, in particular, if $K$ has product structure $K(t)=K_1(t_1)\dots K_d(t_d)$ or if $K$ is rotationally symmetric.
	\item If the non-diagonal terms of $\Delta_Kf(p)$ vanish and, in addition, $\frac{\partial g}{\partial x_j}=0$ then $\Delta_K f$ is the classical Euclidean Laplacian. One case in which this occurs is when $g$ is even.
\end{itemize}

We do not discuss the uniformity of our estimates with respect to $ f $ in a family of functions with bounded norm in the appropriate smoothness space, although the techniques of \cite{gine2006empirical} are applicable in our setting with appropriate additional assumptions.

Section~\ref{sec:corr} discusses another topic that we believe is not addressed in the literature.  We establish that under the conditions we use to prove a central limit theorem for 
\begin{equation*}
	\sqrt{n\epsilon_n^{d+2}}(D_{\epsilon_n,n}f(p)-\Delta_Kf(p))
\end{equation*}
 one also has that the resulting normal random variables have asymptotically vanishing correlations when they are evaluated at two distinct points $p_1$ and $p_2$.  This quantifies the local property of the residual terms in the central limit theorem.

In Section~\ref{sec:bdy} we turn to the situation where the manifold of interest has a boundary or, equivalently, our sampling is done on a subset of the space.  A simple motivating example is to consider one-dimensional situation 
\begin{equation*}
L_\epsilon f(p)=\frac{1}{\epsilon}\mathbb{E}K(\frac{p-X}{\epsilon})(f(X)-f(p))=\frac{1}{\epsilon}\int_{B(p,\epsilon)}(f(y)-f(p))d\lambda(y) 
\end{equation*}
with 
$B(p,\epsilon)=(p-\epsilon,p+\epsilon)\subset \mathbb R$. Here 
we choose  $K=\mathbbm{1}_{(-1,1)}$ and $X$ is uniformly distributed on $(-1,1)$.  If $f \in C^2(-1,1)$ it is not difficult to show that $$\epsilon^{-2}L_\epsilon f(p)  \xrightarrow[\epsilon\to0]{}  \frac16 f''(p)$$ at interior points and that this remains valid for $f$ having vanishing first derivative at the endpoints, see \cite{eigenmaps}. 

To handle the more general boundary structure seen in higher dimensions we introduce, in Definition~\ref{delta_A(p)}, a new Laplacian operator $\Delta_{K,\mathbb{A}(p)}$ that explicitly ``sees'' the boundary $\mathbb{A}(p)$ of the set at the point $p$.  A cancellation condition for convergence of $D_{\epsilon_n,n}f(p)$ to $\Delta_{K,\mathbb{A}(p)}f(p)$ that is valid for quite general kernels is considered in Theorem~\ref{main_thm_for_domains} and simplified in subsequent corollaries. For general kernels we need a very strong vanishing curvature condition on the boundary of our sampling set in order to perform this simplification, see Corollary~\ref{Cor 2.2}, but if $K$ has an even symmetry around an axis that is in some sense normal  to the boundary at $p$ then no such condition is required and we can even permit a singularity at the point $p$, see Corollaries~\ref{cor:Ksymmetries} and~\ref{cor:cornerpts}.  Finally, in Theorem~\ref{thm:genbdy} we show that for kernels that lack this symmetry it is possible that the limit of $D_{\epsilon_n,n}$ contains novel boundary terms involving the tangential component of the gradient at the boundary. 

Thus, 
our work introduces a new boundary-sensitive Laplacian operator and examines its implications for graph Laplacians, demonstrating how boundary effects influence convergence results.
If we define 
\begin{equation}\label{eq-cases}
\Delta_{K,S,g}(p)=\begin{cases}
\Delta_{K }(p)& \text{\ if\ }p\in   \overset{\circ}{S} \\  \Delta_{K,\mathbb{A}(p)}(p)+ \nabla_T f(p)\cdot v_T(p)& \text{\ if\ }p\in \partial S  
\end{cases}
\end{equation}
where the density $g$ is supported in a closed set $S$, with the interior $\overset{\circ}{S}$,
then for any $p\in S$ we have \begin{equation}\label{e-key}
  D_\epsilon f(p)-\Delta_{K,S,g}f(p) \to 0.
\end{equation}
as $\epsilon\to0$ 
under conditions given in above. 
Here $\nabla_T f(p)$, the tangential component of $\nabla f(p)$ at the boundary point $p\in\partial S$ and $ v_T(p)$ is defined by the integral term, involving the curvatures of $\partial S$ at $p$, in the right hand side of \eqref{eqn:genbdy}.
Moreover, \eqref{e-key} implies the weak LLN  \eqref{eq-WLLN}  and the CLT  \eqref{eqn:gkcondits}  using theorems in Section~\ref{approximation}.    

 Our results highlight the role of kernel symmetries, sample point distributions and smoothness of the boundary in determining the limiting operator.
We   note that several special cases of Corollary~\ref{cor:Ksymmetries} occur naturally in practice. One such case is when  $K$ is even, so that $K(-t) = K(t)$ for all $t \in \mathbb{R}^d$, and is dominated by a radial function with finite second moment. A particular case of the latter is when $K$ is the indicator function of a bounded even domain. 
Another important case is when $K(t) = \psi(\|t\|)$ is spherically symmetric around the origin and has a finite second moment. 
In   such a case, if the manifold has $C^2$ smooth boundary, our results imply that the empirical Laplacian converges to the Laplacian \eqref{eq-cases} with Neumann boundary conditions \eqref{eqn:Apcancellation}.

\subsection*{Acknowledgements}
 
The authors are grateful to 
Bobita Atkins, 
Rachel Bailey, 
Ashka Dalal, 
Natalie Dinin, 
Jonathan Kerby-White, 
Tess McGuinness,  
Tonya Patricks, 
Genevieve Romanelli, 
Yiheng Su, 
for interesting and stimulating discussions during preparation of a 
joint paper~\cite{eigenmaps} that is closely connected to this article. 
The last author thanks 
 Raghavendra Venkatraman for 
 helpful  discussions related to the works 
\cite{Venkatraman2023,armstrong2025optimal}. 
We gratefully acknowledge partial support from the NSF through grant NSF DMS~2349433.

\section{Approximation of the Laplacian}\label{approximation}
The operator $D_\epsilon$ has been called the averaging kernel operator. These operators are interesting in general. Koltchinskii and Gine  have studied the closely  related integral operator $$f(p) \mapsto \frac{1}{\epsilon^{d+2}}\mathbb{E}K\left(\frac{p-X}{\epsilon}\right)f(X),$$ obtaining interesting results about the spectrum, under the assumption that the kernel is symmetric and square integrable over the product space and thus Hilbert-Schmidt \cite{koltchinskii2000random}. In this section our main goals are  Theorems~\ref{thm2.4} which is an analogue of the CLT for the deviations of the averaging kernel operator from the empirical graph Laplacian for certain class of functions; Theorem~\ref{thm0.5} which estimates the rate of convergence of a particular Laplacian operator (see Definition~\ref{delta}); Corollary~\ref{combine}, which combines these to show convergence of $D_{\epsilon_n,n}$ to the Laplacian.

\begin{theorem}
\label{thm2.4}
    Let $\{X,X_j\}$ be i.i.d random variables taking values in $\mathbb{R}^d$ with bounded density $g$ that is continuous at $p$ and  $f\in C^2(\mathbb{R}^d)$. 
    Suppose $K:\mathbb{R}^d\to\mathbb{R}$ is a  kernel satisfying
    \begin{equation}\label{eq-Kt4}
 \int_{\mathbb R^d} (\norm{t}^4+\norm{t}) (|K(t)|^4+\abs{K(t)})dt<\infty.  
\end{equation}  
    If  $\epsilon_n>0$ is a sequence such that $\epsilon_n \to 0$ and  
     $n\epsilon_n^d \to \infty$,  
     then
\begin{equation}\label{eq-key}
\sqrt{n\epsilon_n^{d+2}}(D_{\epsilon_n,n}f(p)-D_{\epsilon_n}f(p)) \overset{\text{dist}}\to s\mathcal{Z},
\end{equation} 
as $n$ goes to infinity, where  $\mathcal{Z}$ has the standard normal distribution $N(0,1)$ and 
\begin{equation}\label{eqn:varfactor}
s^2 = g(p) \sum_{i,j=1}^d  \frac{\partial}{\partial x_i}f(p)\frac{\partial}{\partial x_j}f(p)\int_{\mathbb{R}^d}K^2(-t)t_it_j d\lambda(t).
\end{equation}
\end{theorem}

Our proof uses Lyapunov's CLT, which we recall here for the convenience of the reader; a detailed treatment of this result may be found in \cite{billingsley2017probability}.

We note that condition \eqref{eq-Kt4} can be verified using Lemma~\ref{finite_integral_of_moments}.

\noindent
\begin{theorem}[Lyapunov's CLT]\label{Lyapunov} 
  Suppose for each n, the sequence $W_{n,1},\dots,W_{n,r_n}$ are independent and that,
  \begin{equation*}
  \mathbb{E}W_{n,j}=0,\qquad \sigma_{n,j}^2=\mathbb{E}W_{n,j}^2,\qquad s_n^2=\sum_{j=1}^{r_n}\sigma_{n,j}^2.
  \end{equation*}
  If for some $\delta >0$, Lyapunov's condition,
  \begin{equation*}
  \lim_n\frac{1}{s_n^{2+\delta}}\sum_{j=1}^{r_n}\mathbb{E}\abs{W_{n,j}}^{2+\delta}=0
  \end{equation*}
  is satisfied, Then
  $\frac{1}{s_n}\sum_{j=1}^{r_n}W_{n,j}$ converges in distribution to the standard normal random variable as n goes to infinity.
\end{theorem}

We also need Taylor's theorem, for which we use the following notation.

\begin{theorem}{(Taylor's Theorem in Several Variables).\protect{~\cite[Theorem~2 and Corollary~1]{folland2005higher}}}\label{Taylor's Expansion}
    Suppose  $f:\mathbb{R}^d \to \mathbb{R}$ is of class $C^{k+1}$ on an open convex set $S$. If $a \in S$ and  $a+h \in S$, then
    \begin{equation}\label{Taylor}
    f(a+h)=\sum_{\abs{\alpha}\leq k}\frac{\partial f(a)}{\alpha !}h^\alpha + R_{a,k}(h)
    \end{equation}
    where the remainder is given in Lagrange form by 
    \begin{equation}
    R_{a,k}(h)=\sum_{\abs{\alpha}=k+1}\partial^\alpha f(a+ch)\frac{h^\alpha}{\alpha!}, \text{ for some } c \in (0,1).
    \end{equation}\label{eqn:Taylorremainder}
    If $|\partial^\alpha f(x)|\leq M$ for all $x\in S$ and multi-indices $|\alpha|=k+1$ then, with $\|h\|=\sum_1^d|h_j|$, the remainder satisfies the bound
    \begin{equation*}
    \bigl|R_{a,k}(h)\bigr| \leq \frac{M}{(k+1)!}\|h\|^{k+1}.
    \end{equation*}
\end{theorem}

\begin{remark}\label{TaylorRemark}
   In Theorem~\ref{Taylor's Expansion} with $k=2$, equation (\ref{Taylor}) has the form
    \begin{equation*}
        f(a+h)=f(a)+\nabla f(a)\cdot h + \frac{1}{2}h^THh + R_{a,2}(h)
    \end{equation*}
    where H is the Hessian matrix.
\end{remark}

We also record the following, from which the condition~\eqref{eq-Kt4} ensures the first through fourth moments of $|K|^j$ are finite for $j=1,2,3,4$, both when the measure is $d\lambda$ and when it is $gd\lambda$ for the bounded function $g$. This fact is used numerous times in the arguments proving Theorem~\ref{thm2.4}, often without being explicitly mentioned. For example, it justifies our use of standard simplifications of the expectation that are valid only for integrable random variables.

\begin{lemma}\label{NewFiniteMoment}
If $\alpha,\beta,p,r \in [1,\infty)$, $h(t)\geq0$ and 
    \begin{equation*}
    \int\left(h(t)^p+h(t)^r\right)\left(\abs{K(t)}^\alpha+\abs{K(t)}^\beta\right)\,d\lambda(t)<\infty,
    \end{equation*}
then
    \begin{equation*}
    \int\abs{K(t)}^j h(t)^q<\infty, \forall \alpha < j <\beta,  p<q<r.
    \end{equation*}
\end{lemma}

\begin{proof}
The assumptions give $h\in L^p(|K|^\alpha d\lambda)\cap L^r(|K|^\alpha d\lambda)$, so by interpolation also $h\in L^q(|K|^\alpha d\lambda)$.  Similarly $h\in L^q(|K|^\beta d\lambda)$. But then $K\in L^\alpha(|h|^q d\lambda)\cap L^\beta(|h|^q d\lambda)$, so again by interpolation $K\in L^{j}(|h|^qd\lambda)$.\end{proof}

\begin{proof}[Proof of Theorem \ref{thm2.4}]Write
\begin{align}\label{Zn}
&Z_n(p) =\sqrt{n\epsilon^{d+2}}(D_{\epsilon,n}f(p)-D_{\epsilon}f(p))\\\notag
	&=\frac{1}{\sqrt{n\epsilon^{d+2}}}\sum_{j=1}^n\left(K\left(\tfrac{p-X_j}{\epsilon}\right)(f(X_j)-f(p))-\mathbb{E}K\left(\tfrac{p-X_j}{\epsilon}\right)(f(X_j)-f(p))\right)
\end{align}
where the last step used the fact the $X_j$ have the same distribution as $X$. Now, take the first degree Taylor expansion~\eqref{Taylor} of both instances of $f$. It is convenient to write $Y_j=X_j-p, Y=X-p$ and collect the linear and remainder terms so that $Z_n=Z_{n,1}+Z_{n,2}$, with
\begin{align}
    Z_{n,1}(p)&= \frac{1}{\sqrt{n\epsilon^{d+2}}}\sum_{j=1}^n\left(K\left(\frac{-Y_j}{\epsilon}\right)\nabla f(p)\cdot Y_j -\mathbb{E}K\left(\frac{-Y_j}{\epsilon}\right)\nabla f(p)\cdot Y_j\right) \label{eqnZ_n,1}\\
    Z_{n,2}(p)&=\frac{1}{\sqrt{n\epsilon^{d+2}}}\sum_{j=1}^n\left(K\left(\frac{-Y_j}{\epsilon}\right)R_{p,1}(Y_j)-\mathbb{E}K\left(\frac{-Y_j}{\epsilon}\right)R_{p,1}(Y_j)\right) \label{eqnZ_n,2}
\end{align}
in which expressions we have suppressed n and written $\epsilon =\epsilon_n$.
The main term is $Z_{n,1}$ and will be analyzed using the Lyapunov CLT. To do so,
write
\begin{equation}\label{W_n,j}
    W_{n,j}=\frac{1}{\sqrt{n\epsilon^{d+2}}}\left(K\left(\frac{-Y_j}{\epsilon}\right)\nabla f(p)\cdot Y_j -\mathbb{E}K\left(\frac{-Y_j}{\epsilon}\right)\nabla f(p)\cdot Y_j\right)
\end{equation}
and note that for each $n$, the $W_{n,j}$ are i.i.d.\ random variables with respect to $j=1,\dots, n$ and $\Var(W_{n,j})=\frac{1}{n}\Var(Z_{n,1})$. So in notation of Theorem~\ref{Lyapunov} we have, $r_n=n, s_n^2=\Var(Z_{n,1})$. In Lemma~\ref{varZn1 expectationWnj} we find that $\mathbb{E}\abs{W_{n,j}}^4 \leq C_1\norm{\nabla f(p)}^4n^{-2}\epsilon^{-d}$, and that $s_n^2$ converge to $s^2$ given the statement of the theorem. If $s=0$, the convergence of $Z_{n,1}$ to zero in distribution follows immediately. If not, it is apparent that we can apply Lyapounov's CLT with $\delta=2$. Specifically,
\begin{equation*}
\frac{1}{s_n^4}\sum_{j=1}^n\mathbb{E}\abs{W_{n,j}}^4\leq \frac{C_1\norm{\nabla f(p)}^4}{s_n^4n\epsilon^d}
\end{equation*}
and since our hypothesis has $n\epsilon^d \to \infty$ the Lyapunov condition is verified and $\frac{1}{s_n}Z_{n,1}$ converges in distribution to a standard normal random variable.

It remains to show that the terms $Z_{n,2}$  converge to zero. This is done in Lemma~\ref{bound for Z_n,2} where we show $\mathbb{E}Z_{n,2}$ is dominated by a multiple of $\epsilon^2$.
\end{proof}

\begin{lemma}\label{bound for Z_n,2}
    Under the hypothesis of Theorem~\ref{thm2.4} there is a constant $C_2$ , so that the term $Z_{n,2}$ in~\eqref{eqnZ_n,2} satisfies
    \begin{equation*}
    \mathbb{E}Z_{n,2}^2\leq C_2 \sup_{\abs{\alpha}=2}\abs{\partial^\alpha f}^2\epsilon^2
    \end{equation*}
\end{lemma}

\begin{proof}
    By independence of $Y_j=X_j-p$ and the fact that they have the same distribution as $Y=X-p$
\begin{align*}
    \mathbb{E}Z_{n,2}^2
     &=\frac{1}{n\epsilon^{d+2}}\sum_{j=1}^n\mathbb{E}\Big(K(\frac{-Y_j}{\epsilon})R_{p,1}(Y_j)-\mathbb{E}K(\frac{-Y_j}{\epsilon})R_{p,1}(Y_j)\Big)^2\\
     &=\frac1{\epsilon^{d+2}} \mathbb{E}\Big(K(\frac{-Y}{\epsilon})R_{p,1}(Y)-\mathbb{E}K(\frac{-Y}{\epsilon})R_{p,1}(Y)\Big)^2\\
    &\leq \frac{1}{\epsilon^{d+2}}\mathbb{E}\Bigl( K(\frac{-Y}{\epsilon})R_{p,1}(Y)\Bigr)^2
\end{align*}
However, $|R_{p,1}(Y)|\leq C\sup_{|\alpha|=2}|\partial^\alpha f|\|Y\|^2$ and using that g is bounded we obtain
  \begin{align*}
     \mathbb{E}\Bigl( K(\frac{-Y}{\epsilon})\|Y\|^2\Bigr)^2
    &=\int \bigl(K(\frac{-y}{\epsilon})\bigr)^2 \|y\|^4g(y)\,d\lambda(y)\\
      \leq& C_d \|g\|_\infty\epsilon^{d+4}\int \|t\|^4 K^2(-t)\,d\lambda(t). \qedhere
    \end{align*}
\end{proof}

The following lemma can be used to verify condition  \eqref{eq-Kt4}.

\begin{lemma}\label{finite_integral_of_moments}
    Suppose,
    \begin{equation*}
    \abs{K(t)} \leq \frac{C}{\norm{t}^\tau+\norm{t}^\beta}, 
    \end{equation*}
    for some $C\in \mathbb{R}^d$ and for  all $t \in \mathbb{R}^d$. And assume $\alpha,\eta \in \mathbb{N}$ are such that $\alpha\beta - \eta > d$ and we $\alpha\tau-\eta <d$,
    then,
    \begin{equation*}
    \int_{\mathbb{R}^d}\abs{K(-t)}^\alpha \norm{t}^\eta\,d\lambda(t) < \infty
    \end{equation*}
\end{lemma}

\begin{proof}
    We have,
    \begin{align*}
     &    {\int_{\mathbb{R}^d}\abs{K(-t)}^\alpha\norm{t}^\eta\,d\lambda(t)}  \leq C\int_{\mathbb{R}^d}\frac{1}{(\norm{t}^\tau+\norm{t}^\beta)^\alpha}\norm{t}^\eta\,d\lambda(t)
        \\
    &    =C\int_{S^{d-1}}\int_{0}^\infty\frac{1}{(r^\tau+r^\beta)^\alpha}r^\eta r^{d-1}\,dr\,d\sigma(y)
        \\ 
    &    \leq C\int_{S^{d-1}}\left[\int_{0}^1\frac{r^{\eta+d-1}}{r^{\alpha\tau}}\,dr+\int_{1}^\infty\frac{r^{\eta +d-1}}{r^{\alpha\beta}}\,dr\right]\,d\sigma(y) 
    \end{align*}
which are finite by our hypotheses on the exponents.
\end{proof}

\begin{lemma} \label{varZn1 expectationWnj}   
Under the assumption of Theorem~\ref{thm2.4} the term $Z_{n,1}$ in~\eqref{eqnZ_n,1} satisfies
\begin{equation}
\begin{split}
 \Var(Z_{n,1})&=\sum_{1\leq i,j \leq d}\frac{\partial f}{\partial x_i}(p)\frac{\partial f}{\partial x_j}(p)\int_{\mathbb{R}^d}K^2(-t)t_it_jg(p+\epsilon t)\,d\lambda(t)\\
        &\quad +\epsilon^d\left(\sum_{j=1}^d\frac{\partial f}{\partial x_j}(p)\int_{\mathbb{R}^d}K(-t)t_jg(p+\epsilon t)\, d\lambda(t)\right)^2
\end{split}
\end{equation}
    and there is a constant $C_1>0$ such that the term $W_{n,j}$ in~\eqref{W_n,j} satisfies 
    \begin{equation*}
    \mathbb{E}\abs{W_{n,j}}^4 \leq C_1 \frac{\|\nabla f(p)\|^4}{n^2\epsilon^d}
    \end{equation*}
\end{lemma}

\begin{proof}
Using independence of $Y_j=X_j-p$ at the first step and the fact that all $Y_j$ have the distribution of $Y=X-p$ at the second step
\begin{align}
    \Var(Z_{n,1})&=\frac{1}{n\epsilon^{d+2}}\sum_{j=1}^n \Var\Bigl(K(\frac{-Y_j}{\epsilon})\nabla f(p)\cdot(Y_j)-\mathbb{E}K(\frac{-Y_j}{\epsilon})\nabla f(p)\cdot(Y_j)\Bigr) \nonumber \\
   &=\frac{1}{\epsilon^{d+2}} \Var\bigl(K\Bigl(\frac{-Y}{\epsilon}\bigr)\nabla f(p)\cdot Y\Bigr) \label{variance}
   \end{align}
Then compute,
\begin{align}
    \mathbb{E}K^2\Bigl(\frac{-Y}{\epsilon}\Bigr) \bigl(\nabla f(p)\cdot Y \bigr)^2
    &= \int_{\mathbb{R}^d}K^2\Bigl( \frac{p-y}{\epsilon}\Bigr) \biggl( \sum_{j=1}^d (y_j-p) \frac{\partial f}{\partial x_j}(p) \biggr)^2 g(y)\,d\lambda(y) \notag \\
    &=\int_{\mathbb{R}^d}K^2(-t)\Bigl( \sum_{j=1}^d\epsilon t_j \frac{\partial f}{\partial x_j}(p) \Bigr)^2 g(p+\epsilon t) \epsilon^d\,d\lambda(t) \notag\\
    &=\epsilon^{d+2}\sum_{1\leq i,j \leq d}\frac{\partial f}{\partial x_i}(p)\frac{\partial f}{\partial x_j}(p) \int\limits_{\mathbb{R}^d}K^2(-t)t_it_jg(p+\epsilon t)\,d\lambda(t),
\label{first term in variance}
\end{align}
and similarly
\begin{align}
    \mathbb{E}K\Bigl(\frac{-Y}{\epsilon}\Bigr)\nabla f(p)\cdot Y 
    &= \int_{\mathbb{R}^d}K\bigl( \frac{p-y}{\epsilon}\bigr) \sum_{j=1}^d  (y_j-p) \frac{\partial f}{\partial x_j}(p) g (y)\,d\lambda(y) \notag\\
    &=\epsilon^{d+1}\sum_{j=1}^d\frac{\partial f}{\partial x_j}(p) \int_{\mathbb{R}^d}K(-t)t_jg(p+\epsilon t)\,d\lambda(t).
\label{second term in variance}
\end{align}
Substituting~\eqref{first term in variance} and~\eqref{second term in variance} into~\ref{variance} gives the expression for $\Var(Z_{n,1})$.

Turning to the estimate of $\mathbb{E}\abs{W_{n,j}}^2$, we see that it is independent of $j$ and given by
\begin{align*}
	\mathbb{E}|W_{n,j}|^4
	&= \frac1{n^2\epsilon^{2d+4}} \mathbb{E}  \Bigl( K(\frac{-Y}{\epsilon})\nabla f(p)\cdot Y-\mathbb{E}K(\frac{-Y}{\epsilon})\nabla f(p)\cdot Y \Bigr)^4\\
	&\leq  \frac{\|\nabla f(p)\|^4}{n^2\epsilon^{2d+4}} \mathbb{E}  \Bigl( K(\frac{-Y}{\epsilon})  \|Y\|\Bigr)^4\\
	&\leq \frac{\|\nabla f(p)\|^4}{n^2\epsilon^{2d+4}} \int_{\mathbb{R}^d} \Bigl|K\bigl(\frac{p-y}{\epsilon}\bigr)\Bigr|^4 \|y-p\|^4 g(y)\,d\lambda(y)\\
	&\leq  \frac{\|\nabla f(p)\|^4}{n^2\epsilon^{d}}  \int_{\mathbb{R}^d} K(-t) \|t\|^4 g(p+\epsilon t) d\lambda(t).
	\qedhere
\end{align*}
\end{proof} 

\begin{definition}\label{delta}
    Define the operator $\Delta_K$ on the class of $C^2$ functions by,
     \begin{align}   
    \Delta_K f(p):=&\sum_{i,j=1}^d\frac{\partial f} {\partial x_i}(p)\frac{\partial g}{\partial x_j}(p)\int_{\mathbb{R}^d}K(-t)t_it_jd\lambda(t) \\
    &+ \frac{g(p)}{2}\sum_{i,j=1}^d\frac{\partial^2 f}{\partial x_i \partial x_j}(p)\int_{\mathbb{R}^d}K(-t)t_it_jd\lambda(t)\nonumber
    \end{align}
\end{definition}
    
\begin{theorem}
\label{thm0.4}
    Let $X$ be an $\mathbb{R}^d$-valued random variable such that its density $g$ that has bounded $C^1$ norm. If $f:\mathbb{R}^d\to\mathbb{R}$ has bounded $C^2$ norm and $K$  is a kernel satisfying both
\begin{equation}\label{eqn:thm0.4assumption}
\int_{\mathbb{R}^d}\norm{t}^2\abs{K(t)}\,d\lambda(t)<\infty \text{ and } \nabla f(p)\cdot\Bigl(\int_{\mathbb{R}^d}K(-t) t\,d\lambda(t)\Bigr)=0,
\end{equation}
then  $D_\epsilon f(p)\to \Delta_K f(p)$ as $\epsilon\to0$.
\end{theorem}

\begin{proof}
For the first part, using the Taylor expansion {equation}~(\ref{Taylor}) of $f$, we have
    \begin{align*}
    D_{\epsilon}f(p)=&\frac{1}{\epsilon^{d+2}}\mathbb{E}K\left(\frac{p-X}{\epsilon}\right)\left(f(X)-f(p)\right)\\
        =&\frac{1}{\epsilon^{d+2}}\int_{\mathbb{R}^d}K\left(\frac{p-X}{\epsilon}\right)\nabla f(p)\cdot(y-p)g(y)\,d\lambda(y)\\
        &+\frac{1}{\epsilon^{d+2}}\int_{\mathbb{R}^d}K\left(\frac{p-X}{\epsilon}\right)R_{p,1}(y-p)g(y)\,d\lambda(y)
    \end{align*}
Each of  these terms may be simplified by a change of variables. For the first we then use Taylor expansion of $g$, which gives $c=c(p,\epsilon,t)\in(0,1)$ so that
    \begin{align*}
       \lefteqn{ \frac{1}{\epsilon}\int_{\mathbb{R}^d}K(-t)\nabla f(p)\cdot tg(p+\epsilon t)\,d\lambda(t)}\quad&\\
        &=\frac{1}{\epsilon}\int_{\mathbb{R}^d}K(-t)\nabla f(p)\cdot t\left(g(p)+\nabla g(p+c\epsilon t)\cdot \epsilon t\right)\,d\lambda(t)\\ 
        &=\int_{\mathbb{R}^d}K(-t)\left(\nabla f(p)\cdot t\right)\left(\nabla g(p+c\epsilon t)\cdot t\right)\,d\lambda(t).
    \end{align*}
Note that the term involving $g(p)$ was zero because of our cancelation hypothesis on $K$.  
For the second term the change of variables and the expression~\eqref{eqn:Taylorremainder} for the Taylor remainder $R_{p,1}$ give the following, in which  $r=r(p,\epsilon,t)\in(0,1)$ is used for the location at which the Hessian is evaluated in~\eqref{eqn:Taylorremainder}.
    \begin{align*}
        &\frac{1}{\epsilon^{d+2}}\int_{\mathbb{R}^d}K(-t)R_{p,1}(\epsilon t)g(p+\epsilon t)\epsilon^d\,d\lambda(t)\\
        &=\frac{1}{2}\int_{\mathbb{R}^d}K(-t)\bigl( t^THf(p+r\epsilon t)t\bigr) g(p+\epsilon t)\,d\lambda(t)
    \end{align*}

Substituting for these two terms in the  expression for $D_\epsilon f(p)$ and subtracting $\Delta_K$ gives a bound
   \begin{align*}
       \abs{D_{\epsilon}f(p)-\Delta_K f(p)} 
	&\leq \int_{\mathbb{R}^d}\abs{K(-t)}\norm{\nabla f(p)}\norm{t}^2\norm{\nabla g(p+c\epsilon t)-\nabla g(p)}\,d\lambda(t)\\
	+&\frac{1}{2}\int_{\mathbb{R}^d}\abs{K(-t)}\norm{t}^2\norm{g(p)Hf(p)-g(p+\epsilon t)Hf(p+r\epsilon t)}_{op}\,d\lambda(t)
   \end{align*}
from which the result follows from the integrability of $\abs{K(t)}\norm{t}^2$ and the assumptions that $g$ has bounded $C^1$ norm and $f$ has bounded $C^2$ norm.
   \end{proof}

The proof of the next result is similar to the previous one but uses a Taylor expansion of one higher order.  We use the notation $C^{k,\theta}$ for functions for which the $k^\text{th}$ derivatives are H\"older continuous with exponent $\theta\in[0,1]$, with the case $\theta=0$ being continuity and $\theta=1$ being Lipschitz continuity.  Recall that if $f$ has bounded $C^{k,\theta}$ norm we have an estimate of the error $R_{a,k}$ in the order $k$ Taylor expansion as follows.  Using~\eqref{Taylor} for both the order $k-1$ and $k$ expansions and inserting~\eqref{eqn:Taylorremainder} for the remainder $R_{a,k-1}$ lets us write for some $c\in(0,1)$
\begin{align*}	
	R_{a,k}(h)
	&= f(a+h)- \sum_{|\alpha|\leq k}\frac{\partial f(a)}{\alpha !}h^\alpha\\
	&= R_{a,k-1}(h) -  \sum_{|\alpha|=k}\frac{\partial f(a)}{\alpha !}h^\alpha  \\
	&=\sum_{|\alpha|=k}\bigl(\partial^\alpha f(a+ch)-\partial^\alpha f(a)\bigr)\frac{h^\alpha}{\alpha!}
\end{align*}
so that $|R_{a,k}(h)|\leq M \|h\|^{k+\theta}$ for some constant $M$ involving the $C^{k,\theta}$ norm.

   \begin{theorem}\label{thm0.5}
       Let $X$ be an $\mathbb{R}^d$-valued random variable such that its density $g$ has bounded $C^{1,\theta}$ norm  for some $0\leq\theta\leq1$. If $f:\mathbb{R}^d\to\mathbb{R}$ has bounded $C^{2,\theta}$ norm and $K$ is a kernel satisfying both
       \begin{equation*}
       \int_{\mathbb{R}^d} \left(\norm{t}^{2+\theta}+\norm{t}^{3+\theta} \right)\abs{K(t)}\,d\lambda(t)<\infty, \text{ and } 
\nabla f(p)\cdot \Bigl( \int_{\mathbb{R}^d}K(-t)  t\,d\lambda(t) \Bigr)=0,
       \end{equation*}
       then $ \abs{D_\epsilon f(p)-\Delta_K f(p)}=O(\epsilon^\theta)$ as $\epsilon\to0$.
   \end{theorem}

 \begin{proof}
Use  Remark~\ref{TaylorRemark} to write
\begin{equation*}
        D_\epsilon f(p)
	=\frac{1}{\epsilon^{d+2}}\mathbb{E}K\left(\frac{p-X}{\epsilon}\right)(f(X)-f(p)) 
	= A_{n,1}+A_{n,2}+A_{n,3}
    \end{equation*}
where
    \begin{align*}
        A_{n,1}&=\frac{1}{\epsilon^{d+2}}\int_{\mathbb{R}^d}K(\frac{p-y}{\epsilon})\nabla f(p)\cdot(y-p)g(y)d\lambda(y),\\
        A_{n,2}&=\frac{1}{\epsilon^{d+2}}\int_{\mathbb{R}^d}K(\frac{p-y}{\epsilon})\frac{1}{2}(y-p)^THf(p)(y-p)g(y)d\lambda(y),\\
        A_{n,3}&=\frac{1}{\epsilon^{d+2}}\int_{\mathbb{R}^d}K(\frac{p-y}{\epsilon})R_{p,2}(y)g(y)d\lambda(y).\\
    \end{align*}

And after Taylor expansion of $g$ and simplification using the cancellation hypothesis on $K$, we obtain that
\begin{align*}
        A_{n,1}
        &=\frac{1}{\epsilon^{d+2}}\int_{\mathbb{R}^d}K(-t)\nabla f(p)\cdot(\epsilon t)g(p+\epsilon t) \epsilon^d d\lambda(t)\\ 
        &=\frac{1}{\epsilon}\sum_{i=1}^d\frac{\partial f}{\partial x_i} (p)\int_{\mathbb{R}^d}K(-t) t_i g(p+\epsilon t)d\lambda(t)\\
        &=\frac{1}{\epsilon}\sum_{i=1}^d\frac{\partial f}{\partial x_i} (p)\int_{\mathbb{R}^d}K(-t) t_i \Big[g(p)+\nabla g(p)\cdot \epsilon t + R_{p,1}(\epsilon t)\Big]d\lambda(t)\\
        &= \sum_{i,j=1}^d\frac{\partial f}{\partial x_i}(p)\frac{\partial g}{\partial x_j}(p)\int_{\mathbb{R}^d}K(-t)t_it_jd\lambda(t) 
        + O(\epsilon^\theta),
        \end{align*}
where the last line is justified using the bound  $|R_{p,1}(\epsilon t)|\leq M \|\epsilon t\|^{1+\theta}$ and integrating:
    \begin{align*}
     \lefteqn{ \frac{1}{\epsilon}\sum_{i=1}^d\abs{\frac{\partial f}{\partial x_i}(p)\int_{\mathbb{R}^d}K(-t)t_iR_{p,1}(\epsilon t)d\lambda(t)} }\quad &\\
        &\leq M\epsilon^\theta  \sum_i \Bigl|\frac{\partial f}{\partial x_i}(p)\Bigr| \int_{\mathbb{R}^d}\abs{K(-t)}\|t\|^{2+\theta} d\lambda(t).
    \end{align*}
  
The computation for $A_{n,2}$ is similar but does not rely on the cancellation of $K$.
    \begin{align*}
        A_{n,2}
        &=\frac{1}{2\epsilon^{d+2}}\int_{\mathbb{R}^d}K(-t)(\epsilon t)^THf(p)(\epsilon t)g(p+\epsilon t)\epsilon^dd\lambda(t)\\
        &=\frac{1}{2}\int_{\mathbb{R}^d}K(-t)t^THf(p)t\Big[g(p)+\nabla g(p)\cdot \epsilon t +R_{p,1}(t)\Big]d\lambda(t)\\
        &= \frac{g(p)}{2}\sum_{i,j=1}^d\frac{\partial^2 f}{\partial x_i \partial x_j}(p)\int_{\mathbb{R}^d}K(-t)t_it_jd\lambda(t)  + O(\epsilon) + O(\epsilon^{1+\theta})
    \end{align*}
because the second term is bounded by
    \begin{equation*}
    \frac\epsilon 2 \norm{Hf(p)}_{\text{op}}\norm{\nabla g(p)}  \int_{\mathbb{R}^d}\abs{K(-t)}\norm{t}^3d\lambda(t) =O(\epsilon),
    \end{equation*} 
and the third term is bounded by 
    \begin{equation*}
    M \epsilon^{1+\theta} \norm{Hf(p)}_{\text{op}}  \int_{\mathbb{R}^d}\abs{K(-t)}\norm{t}^{3+\theta} d\lambda(t) =O(\epsilon^{1+\theta}).
    \end{equation*}
    
For $A_{n,3}$ we need only use that $g$ is bounded and the estimate for the remainder in the second order Taylor expansion of $f$, which is $\|R_{p,2}(p+\epsilon t)\|\leq M(\epsilon t)^{2+\theta}$,  to obtain
    \begin{align*}
        \abs{A_{n,3}}&\leq \frac{1}{\epsilon^{d+2}}\int_{\mathbb{R}^d}\abs{K(-t)}\abs{R_{p,2}(p+\epsilon t)}\abs{g(p+\epsilon t)}\epsilon^d\,d\lambda(t)\\
        &\leq \frac{M\|g\|_\infty} {\epsilon^2} \int_{\mathbb{R}^d}\abs{K(-t)} \|\epsilon t\|^{2+\theta}\,d\lambda(t)\\
        &=O(\epsilon^\theta).
    \end{align*}
The conclusion follows.
\end{proof}

 \begin{cor}\label{combine}
     Let $\{X,X_j\}$ be i.i.d random variables taking values in $\mathbb{R}^d$ 
and conditions of Theorems~\ref{thm2.4}, \ref{thm0.4} and \ref{thm0.5}     are satisfied for some $0\leq\theta\leq1$. 
If $\epsilon_n>0$ is a sequence such that $n\epsilon_n^{d+2+2\theta}\to0$ and  $n\epsilon_n^d \to \infty$,  then as $n\to\infty$
\begin{equation*}
\sqrt{n\epsilon_n^{d+2}}(D_{\epsilon_n,n}f(p)-\Delta_K f(p)) \overset{d}\to s\mathcal{Z},
\end{equation*} 
where  $\mathcal{Z}$ has the standard normal distribution $N(0,1)$ and 
\begin{equation*}
s^2 = g(p) \sum_{i,j=1}^d  \frac{\partial f}{\partial x_i}(p)\frac{\partial f}{\partial x_j}(p)\int_{\mathbb{R}^d}K^2(-t)t_it_j d\lambda(t).
\end{equation*}
In the case $\theta=0$ it is sufficient that $n\epsilon_n^{d+2}$ is bounded above rather than converging to zero to obtain this result.
 \end{cor} 

 \begin{proof} 
We have that 
\begin{align}
     \lefteqn{\sqrt{n\epsilon_n^{d+2}}(D_{\epsilon_n,n}f(p) - \Delta_K f(p))}&\quad  \label{triangleCorollary}\\
	&=\sqrt{n\epsilon^{d+2}}(D_{\epsilon_n,n}f(p) - D_{\epsilon_n} f(p))
	+\sqrt{n\epsilon_n^{d+2}}(D_{\epsilon_n} f(p) - \Delta_K f(p)) \nonumber 
     \end{align} 
and by Theorem~\ref{thm2.4} the first term converges in distribution to a centered normal random variable with variance $s^2$.
If $\theta\in(0,1]$ we apply Theorem~\ref{thm0.5} to see the second term is bounded by a constant multiple of $\sqrt{ n\epsilon^{d+2+2\theta}}$ which converges to zero by the hypothesis.  In the case $\theta=0$ we may use the weaker hypothesis that $n\epsilon^{d+2}$ is bounded above and Theorem~\ref{thm0.4} to see the same result.  
 \end{proof}

 \section{Asymptotically vanishing correlations}\label{sec:corr}
Under certain assumptions, the limiting deviations $\sqrt{n\epsilon^{d+2}}(D_{\epsilon,n}f(p)-D_{\epsilon}f(p))$ are uncorrelated for distinct $p$ values, as described in the following result.
\begin{theorem}\label{thm:correlation}
Suppose $K$ satisfies the assumptions of Theorem~\ref{thm2.4}. 
If $p_1\neq p_2$ then $$Z_n(p)=\sqrt{n\epsilon_n^{d+2}}(D_{\epsilon_n,n}f(p)-D_{\epsilon_n}f(p))$$ satisfies $\mathbb{E}Z_n(p_1)Z_n(p_2) \to 0$ as $n\to\infty$.
\end{theorem}

\begin{proof}
We write $Z_n(p)=Z_{n,1}(p)+Z_{n,2}(p)$ as in~\eqref{eqnZ_n,1} and~\eqref{eqnZ_n,2} so that
\begin{equation*}
\mathbb{E}Z_n(p_1)Z_n(p_2)=\sum_{1\leq i,j\leq 2}\mathbb{E}Z_{n,i}(p_1)Z_{n,j}(p_2).
    \end{equation*}
Using Lemma~\ref{bound for Z_n,2} and Theorem~\ref{thm2.4} we know 
    $\mathbb{E}Z_{n,1}^2(p)\to s^2(p)$ and $\mathbb{E}Z_{n,2}^2(p)\to 0$
so by Holder's inequality, we see $\mathbb{E}Z_{n,i}(p_1)Z_{n,j}(p_2)\to 0$ when $\{i,j\}\cap\{2\}\neq \emptyset$.
It then remains to show  $\mathbb{E}Z_{n,1}(p_1)Z_{n,1}(p_2) \to 0$.

Recall from~\eqref{eqnZ_n,1} that $Z_{n,1}$ is a sum of $n$ mean-zero terms from the i.i.d.\ random variables $Y_j$:
\begin{equation*}
Z_{n,1}(p)= \frac{1}{\sqrt{n\epsilon^{d+2}}}\sum_{j=1}^n\left(K\Bigl(\frac{-Y_j}{\epsilon}\Bigr)\nabla f(p)\cdot Y_j -\mathbb{E}K\Bigl(\frac{-Y_j}{\epsilon}\Bigr)\nabla f(p)\cdot Y_j\right).
\end{equation*}
In the expectation of the product $Z_{n,1}(p_1)Z_{n,1}(p_2)$ all terms containing $Y_i$ and $Y_j$ with $i\neq j$ will therefore vanish.  The remaining $n$ terms  come from the case $i=j$. Using the fact that all $Y_i$ have the distribution of $Y=X-p$ we compute
\begin{align}
    \lefteqn{\mathbb{E}Z_{n,1}(p_1)Z_{n,1}(p_2)}\quad&\notag\\
    &=\frac{1}{\epsilon^{d+2}}\mathbb{E}K(\frac{p_1-X}{\epsilon})\nabla f(p_1)\cdot (X-p_1)K(\frac{p_2-X}{\epsilon})\nabla f(p_2)\cdot(X-p_2) \notag\\
    &-\frac{1}{\epsilon^{d+2}}\mathbb{E}K(\frac{p_2-X}{\epsilon})\nabla f(p_2)\cdot (X-p_2)\mathbb{E}K(\frac{p_1-X}{\epsilon})\nabla f(p_1)\cdot (X-p_1). \label{eqn:correlationthm1}
    \end{align}
The factors in the second term of~\eqref{eqn:correlationthm1} are bounded as follows:
    \begin{align*}
\lefteqn{ \mathbb{E}\Bigl| K(\frac{p_1-X}{\epsilon})\Bigr| \norm{\nabla f(p_1)}\norm{X-p_1} }\quad &\\
    &= \norm{\nabla f(p_1)} \int\Bigl| K(\frac{p_1-y}{\epsilon})\Bigr| \norm{y-p_1}g(y)\,d\lambda(y)\\
    &= \norm{\nabla f(p_1)} \int\abs{K(-t)}\norm{\epsilon t}g(p_1+\epsilon t)\epsilon^d\,d\lambda(t)\\
    &\leq \|g\|_\infty \norm{\nabla f(p_1)}\epsilon^{d+1}  \int\abs{K(-t)}\norm{t}\,d\lambda(t).
    \end{align*}
so that the whole second term in~\eqref{eqn:correlationthm1} is bounded by
\begin{multline*}
\frac{1}{\epsilon^{d+2}}\mathbb{E}\Bigl| K(\frac{p_2-X}{\epsilon})\Bigr| \norm{\nabla f(p_2)}\norm{X-p_2}\mathbb{E}\Bigl| K(\frac{p_1-X}{\epsilon})\Bigr| \norm{\nabla f(p_1)}\norm{X-p_1} \\
\leq \|g\|_\infty^2 \norm{\nabla f(p_2)}\norm{\nabla f(p_1)}\epsilon^d \Bigl( \int\abs{K(-t)}\norm{t}\,d\lambda(t)\Bigr)^2
\end{multline*}
and converges to $0$ as $\epsilon\to0$.

Now consider the first term of~\eqref{eqn:correlationthm1}. We have
    \begin{align*}
        \lefteqn{\frac{1}{\epsilon^{d+2}}\mathbb{E}K(\frac{p_1-X}{\epsilon})\nabla f(p_1)\cdot (X-p_1)K(\frac{p_2-X}{\epsilon})\nabla f(p_2)\cdot(X-p_2)}\quad&\\
        =&\frac{1}{\epsilon^{d+2}}\int K(\frac{p_1-y}{\epsilon})\nabla f(p_1)\cdot (y-p_1)K(\frac{p_2-y}{\epsilon})\nabla f(p_2)\cdot(y-p_2)g(y)\,d\lambda(y)\\
        \leq& \frac{\|g\|_\infty\|\nabla f(p_1)\|\|\nabla f(p_2)\|}{\epsilon^{d+2}}\int\abs{K(\frac{p_1-y}{\epsilon})}\norm{y-p_1}\abs{K(\frac{p_2-y}{\epsilon})}\norm{y-p_2}\,d\lambda(y)\\ 
	&=\|g\|_\infty\|\nabla f(p_1)\|\|\nabla f(p_2)\| \int |K(-t)|\|t\| \Bigl| K\bigl(\frac{p_2-p_1}{\epsilon} -t\bigr)\Bigr| \Bigl\|\frac{p_2-p_1}{\epsilon} -t\Bigr\| d\lambda(t).
        \end{align*}
This last integral is the convolution $\phi\ast\psi\bigl(\frac{p_2-p_1}{\epsilon}\bigr)$, where $\phi(t)=|K(-t)|\|t\|$ and $\psi(t)=|K(t)|\|t\|$.  Our assumption on the moments of $K$ ensures $\phi$ and $\psi$ are in $L^2(d\lambda)$; by Parseval's identity so are their Fourier transforms $\hat\phi$ and $\hat\psi$, and by the Cauchy-Schwarz inequality then $\hat\phi\hat\psi\in L^1(d\lambda)$.  Using the inverse Fourier transform we can write $\phi\ast\psi\bigl(\frac{p_2-p_1}{\epsilon}\bigr)=(\hat\phi\hat\psi)^\vee\bigl(\frac{p_1-p_2}\epsilon\bigr)$.  Since $p_1\neq p_2$ we have $\bigl\|\frac{p_2-p_1}\epsilon\bigr\|\to\infty$ as $\epsilon\to0$, so applying the Riemann-Lebesgue lemma. we find $\phi\ast\psi\bigl(\frac{p_2-p_1}{\epsilon}\bigr)\to0$ as $\epsilon\to0$.
\end{proof}

\section{Sample points on subsets of $\mathbb{R}^d$}\label{sec:bdy}
Thus far we have assumed that our points $X_j$ are chosen according to a  smooth density $g$ on $\mathbb{R}^d$. However, in many problems it is more natural that the sample points lie on some proper subset $S\subset\mathbb{R}^d$, and the corresponding density $g$ might then be expected to be discontinuous at points of $\partial S$.  In this section we explore the limiting behavior of averaging kernel operators for points with distributions of this type and show that if $p$ is an interior point of $S$ the results are much like those in the previous sections but with a differential operator that depends on the asymptotic structure of $S$ at $p$, while for $p\in\partial S$ it is possible to make smoothness conditions on $\partial S$ that permit characterizing limits of averaging kernels in a similar manner.

\begin{definition}\label{defn:Ap}
  If    $S\subset  \mathbb{R}^d$ and $p \in S$, 
 define 
    \begin{equation}\label{halfspace}
    \mathbb{A}(p)=\lim_{\epsilon\to0}\frac{S-p}{\epsilon}, 
    \end{equation}
    if the limit exists in the sense of  pointwise convergence a.e.\ of the indicator functions.
\end{definition}
\begin{remark}
It is elementary that  $\mathbb{A}(p)$ is a cone with vertex at the origin. The fact that it is defined only up to sets of zero $\lambda$-measure may be interpreted as saying that the set of rays in $\mathbb{A}(p)$ is determined up to a set of zero $(d-1)$-dimensional measure in the unit sphere. 
\end{remark}

We then modify Definition~\ref{delta} of $\Delta_K$ as follows. 
\begin{definition}    \label{delta_A(p)}
Assume $\int_{\mathbb{R}^d}K(t)\norm{t}^2\,d\lambda(t)<\infty$ and
    let $S\subset \mathbb{R}^d$. Suppose $p\in S$ is such that~\eqref{halfspace} exists.
    Define the operator $\Delta_{K,\mathbb{A}(p)}$ on the class of $C^2$ functions on $S$ by,
     \begin{align}   
    \Delta_{K,\mathbb{A}(p)} f(p)
	&= \sum_{i,j=1}^d\frac{\partial f}{\partial x_i}(p)\frac{\partial g}{\partial x_j}(p)\int_{\mathbb{A}(p)}K(-t)t_it_jd\lambda(t)\\ 
    &\quad + \frac{g(p)}{2}\sum_{i,j=1}^d\frac{\partial^2f}{\partial x_i \partial x_j}(p)\int_{\mathbb{A}(p)}K(-t)t_it_jd\lambda(t)\nonumber
    \end{align}
\end{definition}

The following observation will be useful.
\begin{lemma}\label{lem:gradientconditindepofrho}
Suppose a kernel $K$ satisfies $K(-t)\|t\|^2\in L^1(d\lambda(t))$.  If, for some vector $v$, the following limit exists then it is independent of $\rho\in(0,\infty]$:
\begin{equation}
    \lim_{\epsilon\to0^+} \frac{1}{\epsilon}\int_{\frac{S-p}{\epsilon}\cap B(0,\frac{\rho}{\epsilon})}K(-t) v\cdot t\,d\lambda(t).
    \end{equation}
The same is true if we replace the integration region by $\mathbb{A}(p)\cap B\bigl(0,\frac\rho\epsilon\bigr)$. 
\end{lemma}
\begin{proof}
For $0<\rho_1<\rho_2$ we have, as $\epsilon\to0$,
\begin{align*}
\lefteqn{\abs{\frac{1}{\epsilon}\int_{\frac{S-p}{\epsilon}\cap(B(0,\frac{\rho_2}\epsilon)\setminus B(0,\frac{\rho_1}{\epsilon}))} K(-t) v\cdot t\,  d\lambda(t)}} \quad&\\
&\leq \frac{\|v\|}{\rho_1} \int |K(-t)|\|t\| \,\frac{\rho_1}{\epsilon}\, \mathbbm{1}_{\frac{S-p}{\epsilon}\setminus B(0,\frac{\rho_1}{\epsilon})}d\lambda(t)\notag\\
&\leq\frac{\|v\|}{\rho_1} \int\abs{K(-t)}\norm{t}^2\mathbbm{1}_{\frac{S-p}{\epsilon}\setminus B(0,\frac{\rho_1}{\epsilon})}d\lambda(t)
\to 0
\end{align*}
because the integrand is dominated by $\abs{K(-t)}\norm{t}^2\in L^1(d\lambda)$ and converges pointwise to zero.  The proof for integration over $\mathbb{A}(p)\cap B\bigl(0,\frac\rho\epsilon\bigr)$ is the same. 
\end{proof}

Our first result is analogous to  Theorem~\ref{thm0.4}.
\begin{theorem}\label{main_thm_for_domains}
Let $S\subset  \mathbb{R}^d$ and $ p\in S$ be such that~\eqref{halfspace} exists, and $X$ be an $S$-valued random variable with density $g$ that has bounded $C^1$ norm.
Suppose the kernel $K$ satisfies $\int|K(t)|\|t\|^2d\lambda(t)<\infty$ and 
    \begin{equation}
    \nabla f(p)\cdot \Bigl( \frac{1}{\epsilon}\int_{\frac{S-p}{\epsilon}\cap B(0,\frac{\rho}{\epsilon})}K(-t) t\,d\lambda(t) \Bigr)\to 0 \label{gradientCondition}
    \end{equation}
for some and hence all $\rho>0$.
If  $f:S\to\mathbb{R}$ has bounded $C^2$ norm then, as $\epsilon\to0$,
\begin{equation*}
	  D_\epsilon f(p)-\Delta_{K,\mathbb{A}(p)}f(p) \to 0.
	\end{equation*}
\end{theorem}

\begin{proof}
Fixing $\rho$ such that $Hf,\nabla g$ are continuous in $\overline{B(p,\rho)}$ we compute that
\begin{align*}
    D_\epsilon f(p)&=\frac{1}{\epsilon^{d+2}}\mathbb{E}K(\frac{p-X}{\epsilon})(f(X)-f(p)) \notag\\
    &=\frac{1}{\epsilon^{d+2}}\int_SK(\frac{p-y}{\epsilon})(f(y)-f(p))g(y)d\lambda(y) \notag\\
    &=\frac{1}{\epsilon^{d+2}}\int_{\frac{S-p}{\epsilon}}K(-t)(f(p+\epsilon t)-f(p))g(p+\epsilon t)\epsilon^d d\lambda(t).
\end{align*}
However, as $\epsilon\to0$ we have
\begin{align*}
\lefteqn{\abs{\frac{1}{\epsilon^2}\int_{\frac{S-p}{\epsilon}\setminus B(0,\frac{\rho}{\epsilon})}
K(-t)(f(p+\epsilon t)-f(p))g(p+\epsilon t)d\lambda(t)}}\quad&\notag\\
&\leq \frac{2}{\rho^2}\|f\|_\infty\|g\|_\infty \int \abs{K(-t)}\frac{\rho^2}{\epsilon^2}\mathbbm{1}_{\frac{S-p}{\epsilon}\setminus B(0,\frac{\rho}{\epsilon})}d\lambda(t)\notag\\
&\leq\frac{2}{\rho^2} \|f\|_\infty\|g\|_\infty \int\abs{K(-t)}\norm{t}^2\mathbbm{1}_{\frac{S-p}{\epsilon}\setminus B(0,\frac{\rho}{\epsilon})}d\lambda(t)
\to 0,
\end{align*}
so that, as $\epsilon\to0$,
\begin{equation}
D_\epsilon f(p)
=\frac{1}{\epsilon^{2}}\int_{\frac{S-p}{\epsilon}\cap B(0,\frac\rho\epsilon)}K(-t)(f(p+\epsilon t)-f(p))g(p+\epsilon t)  d\lambda(t) + o(1)\label{eqn:main_thm_for_domains1}
\end{equation}

Using the second order Taylor series for $f$ and the first order Taylor series for $g$ we find there are  $r_t=p+r\epsilon t$ for some $r=r(p,\epsilon,t)\in(0,1)$ and $s_t=p+s\epsilon t$ for some $s=s(p,\epsilon,t)\in(0,1)$ so that
\begin{align*}&
{(f(p+\epsilon t)-f(p))g(p+\epsilon t)}  = \bigl( \nabla f(p)\cdot \epsilon t +\frac12\epsilon^2 t^T\cdot Hf(r_t)\cdot t\bigr)g(p+\epsilon t)\\
&= \epsilon g(p) \nabla f(p)\cdot t + \epsilon^2 \bigl( \nabla f(p)\cdot  t \bigr)\bigl(\nabla g(s_t)\cdot t\bigr) + \frac12 \epsilon^2 \bigl( t^T\cdot Hf(r_t)\cdot t\bigr)g(p+\epsilon t).
\end{align*}

We substitute these into~\eqref{eqn:main_thm_for_domains1} to obtain three integrals.  If we then subtract $\Delta_{K,\mathbb{A}(p)}f(p)$ we may write $D_\epsilon- \Delta_{K,\mathbb{A}(p)}=B_1+B_2+B_3+o(1)$, each of which will be seen to converge to zero as $\epsilon\to0$.  The first is
\begin{equation*}
B_1=\frac{g(p)}{\epsilon}\int_{\frac{S-p}{\epsilon}\cap B(0,\frac\rho\epsilon)}K(-t)\nabla f(p)\cdot t d\lambda(t),
\end{equation*} 
which converges to zero by~\eqref{gradientCondition}.
Writing $S_{p,\rho,\epsilon}=\frac{S-p}{\epsilon}\cap B(0,\frac\rho\epsilon)$, we have for the second term
\begin{align*}
|B_2|&=\frac{1}{2} \biggl| \int_{S_{p,\rho,\epsilon}} K(-t)\bigl( t^THf(r_t)t\bigr) g(p+\epsilon t) d\lambda -\int_{\mathbb{A}(p)}K(-t)\bigl( t^THf(p)t\bigr) g(p) d\lambda\biggr|\\
&\leq \frac{1}{2}\int \abs{K(-t)}\norm{t}^2\norm{Hf(r_t)g(p+\epsilon t)\mathbbm{1}_{S_{p,\rho,\epsilon}}-Hf(p)g(p)\mathbbm{1}_{\mathbb{A}(p)}}_{\text{op}}d\lambda(t),
\end{align*}
while the third term satisfies
\begin{align*}
|B_3|&=\int_{S_{p,\rho,\epsilon}} K(-t)\nabla f(p)\cdot t \nabla g(s_t)\cdot t d\lambda-\int_{\mathbb{A}(p)}K(-t)\nabla f(p)\cdot t \nabla g(p) \cdot t d\lambda\\
&\leq \norm{\nabla f(p)}\int \abs{K(-t)}\norm{t}^2\abs{\nabla g(s_t)\mathbbm{1}_{S_{p,\rho,\epsilon}}-\nabla g(p)\mathbbm{1}_{\mathbb{A}(p)}}d\lambda(t).
\end{align*}
Then the integrands for both $B_2$ and $B_3$ converge to zero $\lambda$-a.e.\ as $\epsilon\to0$ using continuity of $Hf$, $g$ and $\nabla g$ in $\overline{B(p,\rho)}$ and the fact that $\mathbbm{1}_{S_{p,\rho,\epsilon}}\to\mathbbm{1}_{\mathbb{A}(p)}$ a.e.\  as $\epsilon\to0$.  Moreover, the integrands are dominated by $M\abs{K(-t)}\norm{t}^2\in L^1(d\lambda)$ for some $M$ because $f$ has bounded  $C^2$ norm and $g$ has bounded $C^1$ norm. The conclusion of the theorem then follows by dominated convergence. 
\end{proof}

Since $\frac{S-p}\epsilon$ has limit $\mathbb{A}(p)$ in the sense of Definition~\ref{defn:Ap}  it is natural to wonder whether we can replace the former with the latter in the condition~\eqref{gradientCondition}.   A simple case where we can do so is when $p$ is interior to $S$ so $\mathbb{A}(p)=\mathbb{R}^d$.

\begin{cor}
\label{Cor2.1}
Let $p$ be an interior point of $S\subset  \mathbb{R}^d$. The conclusion of Theorem~\ref{main_thm_for_domains} remains valid if we replace the condition~\eqref{gradientCondition} with
\begin{equation*}
    \nabla f(p)\cdot\Bigl( \int_{\mathbb{R}^d}K(t)\cdot t\,d\lambda(t)\Bigr)=0.
\end{equation*}
 \end{cor}
\begin{proof}
Choose $\rho >0$ so $B(p,\rho)\subset S$ and therefore $\frac{S-p}\epsilon \cap B\bigl(0,\frac\rho\epsilon\bigr)=B\bigl(0,\frac\rho\epsilon\bigr)$.  The hypothesis gives
\begin{equation*}
	\lim_{\epsilon\to0^+} \frac{1}{\epsilon}\int_{B(0,\frac{\rho}{\epsilon})}K(-t)\nabla f(p) \cdot t d\lambda(t) =0
\end{equation*}
in the case $\rho=\infty$, so by Lemma~\ref{lem:gradientconditindepofrho} we have~\eqref{gradientCondition} and the result follows.
\end{proof}

\begin{remark}
Since Corollary~\ref{Cor2.1} is the case $\mathbb{A}(p)=\mathbb{R}^d$ and $\Delta_{K,\mathbb{A}(p)}f(p)=\Delta_Kf(p)$, it extends Theorem~\ref{thm0.4}. This is not suprising because the result is local.
\end{remark}

Turning to the question of whether we can replace $\frac{S-p}\epsilon$ with $\mathbb{A}(p)$ in condition~\eqref{gradientCondition} when $p\in\partial S$, we observe from Lemma~\ref{lem:gradientconditindepofrho} that this is a matter of estimating
\begin{equation*}
\frac1\epsilon \int_{B(0,\frac\rho\epsilon)} K(-t) \nabla f(p)\cdot t \bigl( \mathbbm{1}_{\frac{S-p}\epsilon}-\mathbbm{1}_{\mathbb{A}(p)}\bigr) d\lambda(t).
\end{equation*}
It is convenient to write this in polar coordinates $t=(r,\sigma)$ with $d\sigma$ the normalized measure on the unit sphere $\mathbb{S}$. Since $\mathbb{A}(p)$ is a cone its characteristic function can be written as a function of $\sigma$, and we also use $\mathbbm{1}_{\frac{S-p}\epsilon}(r\sigma)=\mathbbm{1}_{\frac{S-p}{r\epsilon}}(\sigma)$. Our expression becomes
\begin{equation}\label{eqn:radialgradest}
\int_0^{\frac\rho\epsilon} r^d \Bigl( \frac1\epsilon \int_{\mathbb{S}} K(-r\sigma) v\cdot\sigma  \bigl( \mathbbm{1}_{\frac{S-p}{r\epsilon}}(\sigma)-\mathbbm{1}_{\mathbb{A}(p)}(\sigma)\bigr) d\sigma \Bigr) dr.
\end{equation}
Observe that $r\leq\frac\rho\epsilon$ so $r\epsilon\leq \rho$ and we can assume $\rho$ is small. Moreover,  $\frac{S-p}{r\epsilon}\to\mathbb{A}(p)$ a.e.\ on the finite measure space $(\mathbb{S},d\sigma)$ so the set where this function is non-zero has small measure when $\rho$ is small.  However, this cannot be used to gain a factor of $\epsilon$, so any positive result must either require greater regularity of $\partial S$ at $p$ or that $\partial S$ is very thin at $p$, such as occurs at a cusp.

What is more, even knowing $\mathbbm{1}_{\frac{S-p}{r\epsilon}}-\mathbbm{1}_{\mathbb{A}(p)}$ is only non-zero on a small set is not sufficient for a result unless we also know $K$ cannot be too large on this set.  In all of the following results we achieve this by, in essence, assuming $K(r\sigma)$ can be replaced with $\sup_\sigma |K(r\sigma)|$ without affecting its integrability. 

\begin{cor}
\label{Cor 2.2}
Let $S\subset\mathbb{R}^d$ have $C^2$ boundary in a neighborhood of $p\in\partial S$ and suppose $\partial S$ has zero principal curvatures at $p$.  Further assume that the kernel $K$ satisfies  $\abs{K(t)}\leq {h(\|t\|)}$ for some  $h(r)$ with $\int r^{d+1}h(r)dr<\infty$.
Then the conclusion of Theorem~\ref{main_thm_for_domains} remains valid if the condition~\eqref{gradientCondition} is replaced with 
 \begin{equation}\label{eqn:Apcancellation}
       \nabla f(p)\cdot\Bigl( \int_{\mathbb{A}(p)}K(-t) t d\lambda(t)\Bigr)=0.
 \end{equation}
\end{cor}

\begin{proof}
The integrability hypothesis on $h$ gives $\int \abs{K(t)}\norm{t}^2<\infty$, so we need only verify condition~\eqref{gradientCondition} in Theorem~\ref{main_thm_for_domains} to establish the result. 
Using our cancellation hypothesis and Lemma~\ref{lem:gradientconditindepofrho} this follows if for any $\delta>0$ there is $\rho>0$ so
\begin{equation*}
\limsup_{\epsilon\to0} \frac1\epsilon \biggl| \int_{B(0,\frac\rho\epsilon)} K(-t) \nabla f(p)\cdot t \bigl( \mathbbm{1}_{\frac{S-p}\epsilon}-\mathbbm{1}_{\mathbb{A}(p)}\bigr) d\lambda(t)\biggr| <\delta
\end{equation*}
because this limit is independent of $\rho$.  Rewriting as in~\eqref{eqn:radialgradest} and  bounding $K$ by $h$ we arrive at the sufficient condition
\begin{equation}\label{eqn1:Cor 2.2}
	\lim_{\epsilon\to0^+} \int_0^{\frac\rho\epsilon} r^d h(r) \Bigl( \frac1\epsilon \int_{\mathbb{S}} \bigl| \mathbbm{1}_{\frac{S-p}{r\epsilon}}(\sigma)-\mathbbm{1}_{\mathbb{A}(p)}(\sigma)\bigr| d\sigma \Bigr) dr<\delta.
\end{equation}

Now choose $\rho>0$ so that $Hf$ and $\nabla g$ are continuous on $B(p,\rho)\cap S$.  Observe that $\mathbb{A}(p)$ is a half-space (up to a measure zero set) and choose coordinates $(x',x_d)$ for $\mathbb{R}^d$ so the origin is at $p$ and $\mathbb{A}(p)=\{(x',x_d):x_d>0\}$.  Now $\partial(S-p)$ is the graph $\{(x',\gamma(x')\}$ of a $C^2$ function $\gamma$ on a neighborhood of $0$ and $\gamma$ vanishes to second order at the origin by our curvature assumption~\cite[Appendix C]{evans2022partial} and \cite[Section 2.2]{Kobayashi}.  So we see from the Taylor remainder estimate~\eqref{eqn:Taylorremainder} that, by shrinking $\rho$ if necessary, we may assume $|\gamma(x')|\leq \delta \|x'\|^2$ on $B(0,\rho)$. Then for $r\epsilon<\rho$
\begin{align*}
	\bigl| \mathbbm{1}_{\frac{S-p}{r\epsilon}}(\sigma)-\mathbbm{1}_{\mathbb{A}(p)}(\sigma)\bigr|
	&=\bigl| \mathbbm{1}_{S-p}(r\epsilon \sigma)-\mathbbm{1}_{\mathbb{A}(p)}(r\epsilon \sigma ) \bigr| \\
	&\leq \mathbbm{1}_{\{|(r\epsilon \sigma)_d|\leq \delta (r\epsilon)^2\}} (r\epsilon \sigma) \\
	&=  \mathbbm{1}_{\{|(\sigma)_d|\leq \delta r\epsilon \}} (\sigma)
\end{align*}
and it is not difficult to conclude
\begin{equation*}
 \int_{\mathbb{S}} \bigl| \mathbbm{1}_{\frac{S-p}{r\epsilon}}(\sigma)-\mathbbm{1}_{\mathbb{A}(p)}(\sigma)\bigr| d\sigma
< \frac\pi2 \delta r\epsilon.
\end{equation*}
This bounds~\eqref{eqn1:Cor 2.2} by $\delta\int_0^\infty r^{d+1}h(r)dr$, and the result follows.
\end{proof}

The requirement that the curvature at $p$ is zero is very strong. It may be removed in the case that $K$ is  an even function  on a neighborhood of the origin in the directions spanning $\partial\mathbb{A}(p)-p$.

\begin{cor}\label{cor:Ksymmetries}
Let $S\subset\mathbb{R}^d$ have $C^2$ boundary in a neighborhood of $p\in\partial S$. Then $\partial\mathbb{A}(p)-p$ is a codimension $1$ subspace of $\mathbb{R}^d$. Instead of assuming vanishing principal curvatures at $p$, we assume that the kernel satisfies  $\abs{K(t)}\leq {h(\|t\|)}$ for some  $h(r)$ with $\int r^{d+1}h(r)dr<\infty$ and $K(-t)=K(t)$ for all $t\in\partial\mathbb{A}(p)-p$. 
Then the conclusion of Theorem~\ref{main_thm_for_domains} remains valid if the condition~\eqref{gradientCondition} is replaced with~\eqref{eqn:Apcancellation}   
\end{cor}
\begin{proof}
As in the proof of Corollary~\ref{Cor 2.2} it suffices to show that for $\delta>0$ there is $\rho>0$ so
\begin{equation}\label{eqn1:Ksymmetries}
\limsup_{\epsilon\to0} \frac1\epsilon  \biggl| \int_{B(0,\frac\rho\epsilon)} K(-t) \nabla f(p)\cdot t \bigl( \mathbbm{1}_{\frac{S-p}\epsilon}-\mathbbm{1}_{\mathbb{A}(p)}\bigr) d\lambda(t) \biggr|<\delta.
\end{equation}
It is again convenient to use the coordinate system from the proof of Corollary~\ref{Cor 2.2} in which $\partial(S-p)$ is the graph $\{(x',\gamma(x')\}$ with $x_d\geq\gamma(x')$ for $(x',x_d)\in S-p$, while $\mathbb{A}(p)=\{(x',x_d):x_d\geq0\}$.  Taylor's theorem provides that $\gamma(x')=(x')^T \cdot H\gamma \cdot x'+o(\|x'\|^2)$ at the origin. We may therefore find $\rho$ so on $B(0,\rho)$ we have $|\gamma(x')-(x')^T \cdot H\gamma \cdot x|<\delta\|x'\|^2$.

At this point we define $J^+=\{0<x_d<(x')^T\cdot H\gamma\cdot x' \}$ and $J^-=\{(x')^T\cdot H\gamma\cdot x'<x_d<0\}$ and observe that for both $J^+$ and $J^-$ we have
\begin{equation}\label{eqn2:Ksymmetries}
\frac1\epsilon \int_{B(0,\frac\rho\epsilon)}   K(-(x',x_d)) \nabla f(p)\cdot (x',x_d) \mathbbm{1}_{\frac {J^{\pm}} \epsilon}(x',x_d) d\lambda(x',x_d) = 0
	\end{equation}
because the integration region, the kernel $K$ and the indicator functions are symmetric under the map $x'\mapsto -x'$ while the factor  $\nabla f(p)\cdot (x',x_d)$ is antisymmetric under this map.   It follows that we can add or subtract $\mathbbm{1}_{\frac {J^{\pm}} \epsilon}(x',x_d)$ to the characteristic functions in the integrand of~\eqref{eqn1:Ksymmetries} without changing the integral.  However, the characteristic function term in that integral then satisfies
\begin{align*}
\lefteqn{ \bigl| \mathbbm{1}_{\frac{S-p}{\epsilon}}(x',x_d) - \mathbbm{1}_{\mathbb{A}(p)}(x',x_d) +\mathbbm{1}_{\frac {J^+}\epsilon}(x',x_d) - \mathbbm{1}_{\frac {J^-} \epsilon}(x',x_d)  \bigr| }\quad &\\
&=\bigl| -\mathbbm{1}_{0<y_d<\gamma(y')}(\epsilon x',\epsilon x_d) +  \mathbbm{1}_{\gamma(y')<y_d<0}(\epsilon x',\epsilon x_d) + \mathbbm{1}_{J^+}(\epsilon x', \epsilon x_d) - \mathbbm{1}_{J^-}(\epsilon x', \epsilon x_d) \bigr|\\
&\leq \mathbbm{1}_{|y_d-(y')^T\cdot H\gamma\cdot y' |<\delta\|y'\|^2}(\epsilon x',\epsilon x_d)   \\
&= \mathbbm{1}_{|x_d-\epsilon (x')^T\cdot H\gamma\cdot x' |<\delta \epsilon\|x'\|^2}(x', x_d).
\end{align*}
on $B(0,\rho)$ because  $|\gamma(x')-(x')^T \cdot H\gamma \cdot x|<\delta\|x'\|^2$.
Then~\eqref{eqn1:Ksymmetries} may be seen to be bounded by $\delta\int_0^\infty r^{d+1}h(r)$ as in the proof of Corollary~\ref{Cor 2.2}, which completes the proof.
\end{proof}

A similar idea to that used in Corollary~\ref{cor:Ksymmetries} may be applied to the case where $\partial S$ is smooth in a punctured neighborhood of $p$ but only continuous at $p$, provided that both $S$ and the kernel $K$ are even with respect to a single symmetry axis.

\begin{cor}\label{cor:cornerpts}
Let $S\subset\mathbb{R}^d$ have boundary that is $C^2$ in a punctured neighborhood of $p$ and continuous at $p$. Assume there is $\rho>0$ so that $S-p\cap B(0,\rho)$ is even around an axis of symmetry $n_p$ in the following sense: for all vectors $x'$ orthogonal to $n_p$ we have $x'\in (S-p)\cap B(0,\rho)$ if and only if $-x'\in (S-p)\cap B(0,\rho)$.  For the kernel $K$, suppose there is $h$ so $|K(t)|\leq h(\|t\|)$ and $\int r^{d+1}h(r)dr<\infty$, and that for all $x'$ orthogonal to $n_p$ we have $K(x')=K(-x')$.    Then the conclusion of Theorem~\ref{main_thm_for_domains} remains valid if the condition~\eqref{gradientCondition} is replaced with~\eqref{eqn:Apcancellation}. We note that in this case it is possible that $\mathbb{A}(p)$ is measure zero and this latter condition is trivial; this occurs if $S$ has a cusp at $p$.
\end{cor}
\begin{proof}
Taking coordinates $(x',x_d)$ with origin $p$ so $x_d$ is in the direction of $n_p$ we obtain that, on a neighborhood of $p$, we have $\partial (S-p)=\{(x',\gamma(x')\}$ and $S-p=\{(x',x_d):x_d>\gamma(x')\}$ where $\gamma$ is $C^2$ on each radial interval from the origin in the neighborhood. Take the one-sided Taylor expansion of $\gamma$ at the origin along each ray $x'=r\sigma'$, where $r>0$ and $\sigma'$ is a unit vector in the plane orthogonal to $n_p$ and write it as $\gamma(r\sigma')=D_{\sigma'}r+ H_{\sigma'}r^2+o(r^2)$, where it is possible that $D_{\sigma'}=+\infty$, meaning that $S$ degenerates in this direction and there are no $x_d$ so $(x',x_d)\in S-p$ for this choice of $x'$. 
The symmetry yields $D_{-\sigma'}=D_{\sigma'}$ and $H_{-\sigma'}=H_{\sigma'}$ and it is easily seen that we can write $\mathbb{A}(p)-p=\{(r\sigma',x_d):x_d\geq rD_{\sigma'}\}$ for the cone.

As in Corollaries~\ref{Cor 2.2} and~\ref{cor:Ksymmetries}, given $\delta>0$ it suffices to show there is $\rho>0$ so 
\begin{equation*}
\lim_{\epsilon\to0} \frac1\epsilon \biggl|\int_{B(0,\frac\rho\epsilon)} K(-t) \nabla f(p)\cdot t \bigl( \mathbbm{1}_{\frac{S-p}\epsilon}-\mathbbm{1}_{\mathbb{A}(p)}\bigr) d\lambda(t)\biggr| <\delta.
\end{equation*}
For $(x',x_d)\in B(0,\rho)$ we write the indicator functions in the integrand as
\begin{align*}
\lefteqn{ \mathbbm{1}_{x_d>\gamma(x')}(\epsilon x',\epsilon x_d)  - \mathbbm{1}_{x_d> D_{\sigma'}r }(\epsilon x',\epsilon x_d)}\quad &\\
&=  \bigl( \mathbbm{1}_{x_d>\gamma(x')}(\epsilon x',\epsilon x_d) - \mathbbm{1}_{x_d> D_{\sigma'}r + H_{\sigma'}r^2}(\epsilon x',\epsilon x_d) \bigr)\\
&\quad + \bigl( \mathbbm{1}_{x_d> D_{\sigma'}r + H_{\sigma'}r^2}(\epsilon x',\epsilon x_d) - \mathbbm{1}_{x_d> D_{\sigma'}r }(\epsilon x',\epsilon x_d)\bigr).
\end{align*}
Substituting the first of these into the integral we obtain the desired estimate by the argument in Corollary~\ref{Cor 2.2} because the region between the indicator functions is $o(\|x'\|^2)$.   Substituting the second of these into the integral gives zero by the symmetry argument in the proof of Corollary~\ref{cor:Ksymmetries}.  The result follows.
\end{proof}

To complete our discussion, we consider what can happen at $p\in\partial S$ if we have neither the vanishing principal curvature discussed in Corollary~\ref{Cor 2.2} nor the symmetry discussed in Corollaries~\ref{cor:Ksymmetries} and~\ref{cor:cornerpts}.  For this to be interesting we must suppose that we have the cancellation condition discussed in those results, as this would otherwise cause $D_\epsilon f(p)$ to diverge as $\epsilon\to0$. 

The following theorem, which implies \eqref{eq-cases}, is one of our main results. 

\begin{theorem}\label{thm:genbdy}
Let $S\subset\mathbb{R}^d$ have $C^2$ boundary in a neighborhood of $p\in\partial S$, so $
\partial\mathbb{A}(p)$ is the tangent plane at $p$.  Using $x'$ as coordinate on $
\partial\mathbb{A}(p)-p$ let $H$ denote the Hessian at $p$ of any parametrization of $\partial S-p$ with respect to $x'$ on a neighborhood of $p$.

Assume that the kernel $K$ is continuous on $\mathbb{R}^d$ and satisfies $|K(t)|\leq h(\|t\|)$ for some $h(r)$ with $\int_0^\infty r^{d+1}h(r)dr<\infty$. If $f:S\to\mathbb{R}$ has bounded $C^2$ norm and the cancellation condition~\eqref{eqn:Apcancellation} holds then, as $\epsilon\to0^+$,
\begin{equation}\label{eqn:genbdy}
	D_\epsilon f(p)\longrightarrow   \Delta_{K,\mathbb{A}(p)}f(p) - \int\limits_{\partial\mathbb{A}(p)-p}   K(-x')\bigl(\nabla f (p)\cdot x'\bigr)\, \bigl((x')^T\cdot H \cdot x'\bigr)  d\lambda_{d-1}(x')
\end{equation}
where $\lambda_{d-1}$ is Lebesgue measure on the $d-1$ dimensional subspace $\partial\mathbb{A}(p)-p$.
\end{theorem}
\begin{proof}
We use the proof of Corollary~\ref{cor:Ksymmetries}. All that has changed is that the term 
\begin{equation}\label{eqn1:genbdy}
\frac1\epsilon \int_{B(0,\frac\rho\epsilon)}   K(-(x',x_d)) \nabla f(p)\cdot (x',x_d) \bigl( \mathbbm{1}_{\frac {J^-}\epsilon}(x',x_d) -  \mathbbm{1}_{\frac {J^+}\epsilon}(x',x_d) \bigr) d\lambda(x',x_d)
	\end{equation}
is no longer zero by symmetry, and we have $H\gamma(0)=H$.  We must show that this term converges to the right side of~\eqref{eqn:genbdy}.

It is helpful to write~\eqref{eqn1:genbdy} in spherical polar coordinates, using $r$ for the radius and $\sigma$ for a point in the unit sphere $\mathbb{S}^{d-1}$. Then decompose the latter using an angular coordinate $\theta$ to the $x_d$-axis and a coordinate $\sigma'$ on the sphere $\mathbb{S}^{d-1}\cap(\partial\mathbb{A}(p)-p)$, which by a slight abuse of notation we label $\mathbb{S}^{d-2}$ with coordinate $\sigma'$. For convenience we arrange that $\theta\in[-\frac\pi2,\frac\pi2]$ is $\pm\frac\pi2$ in the positive and negative $x_d$ directions and zero on $\mathbb{S}^{d-2}$, so  $\sigma=(\sigma',\theta)$ and the Jacobian is $\cos\theta$.  To  shorten the notation we write $Q(\sigma')=(\sigma')^T\cdot H \cdot \sigma'$ for the quadratic form and note that
\begin{align*}
	\frac {J^+}\epsilon&=\{(\sigma',\sigma_d): 0<\sigma_d<r\epsilon Q(\sigma')\}\\
	\frac {J^-}\epsilon&=\{ (\sigma',\sigma_d): r\epsilon Q(\sigma')<\sigma_d<0\}.
\end{align*}

This provides that the quantity in~\eqref{eqn1:genbdy} is
\begin{align*}
\lefteqn{ \frac1\epsilon 
\int_0^{\frac\rho\epsilon} \int_{\mathbb{S}^{d-1}}   K(-r\sigma )) \nabla f(p)\cdot (r\sigma) \bigl( -\mathbbm{1}_{0<\sigma_d<r \epsilon  Q(\sigma')}+\mathbbm{1}_{r \epsilon  Q(\sigma')<\sigma_d<0}\bigr)  (\sigma) d\sigma\,  r^{d-1}dr } & \\
&= 
\int_0^\infty 
\int_{\mathbb{S}^{d-2}} \frac1{r\epsilon} 
\int_{-\pi}^\pi K(-r(\sigma',\theta))\nabla f (p)\cdot(\sigma',\theta) \chi_{\sigma'}(\theta)  \cos\theta\, d\theta\, d\sigma'\, \mathbbm{1}_{[0,\frac\rho\epsilon]}(r) r^{d+1} dr.
\end{align*}
where
\begin{equation*}
	\chi_{\sigma'}= \mathbbm{1}_{\{\theta:\arcsin(r\epsilon Q(\sigma'))<\theta<0\}} - \mathbbm{1}_{\{\theta: 0<\theta<\arcsin(r\epsilon Q(\sigma'))\}}.
\end{equation*}
For fixed $\sigma'$, the $\theta$ integral may be viewed as integrating against $\frac1{r\epsilon}$ times the characteristic function of the interval between $0$ and $\arcsin(r\epsilon Q(\sigma'))$  with respect to $\cos\theta\,d\theta$, which converges (weakly) to a Dirac mass of weight $-Q(\sigma')$ at $\theta=0$.  Since the integrand is continous we conclude that, as $\epsilon\to0$,
\begin{multline*}
	\frac1{r\epsilon} \int_{-\pi}^\pi K(-r(\sigma',\theta))\nabla f (p)\cdot(\sigma',\theta) \chi_{\sigma'} \cos\theta\, d\theta\\
	\to -Q(\sigma')  K(-r(\sigma',0))\nabla f (p)\cdot(\sigma',0) = -Q(\sigma')   K(-r\sigma')\nabla f (p)\cdot\sigma'
	\end{multline*}
pointwise for each $r,\sigma'$.  Moreover, these are dominated by $h(r)\|\nabla f(p)\|$.  But integrating $h(r)$ in the above gives
\begin{equation}\label{eqn2:genbdy}
	\int_0^\infty \int_{S^{d-2}} |h(r)| r^{d+1}\,dr.
	\end{equation}
As this is the integral of a radial function over a codimension-$1$ subspace, it is independent of the subspace. Integrating the resulting constant with respect to the natural (finite) measure on the Grassmanian of such subspaces then gives $\int_{\mathbb{R}^d}h(r) r^{d+1}dr$, which is finite by hypothesis, so we conclude that~\eqref{eqn2:genbdy} is finite and our dominating function is integrable.  Thus by the dominated convergence theorem the quantity in~\eqref{eqn1:genbdy} converges to
\begin{equation*}
	\int_0^\infty \int_{\mathbb{S}^{d-2}} -Q(\sigma')   K(-r\sigma')\nabla f (p)\cdot\sigma'\, d\sigma' \, r^{d+1}\,dr.
	\end{equation*}
Here the region of integration is the subspace $\partial\mathbb{A}(p)-p$; writing this via $x'=r\sigma'$, absorbing $r^{d-2}$ into the corresponding Jacobian, and inserting $Q(\sigma')=(\sigma')^T\cdot H\cdot \sigma'=r^{-2} (x')^T\cdot H\cdot x'$  gives that the limit of the term from~\eqref{eqn1:genbdy} is given by the right side of~\eqref{eqn:genbdy}, 
which completes the proof.
\end{proof} 
\newcommand{\etalchar}[1]{$^{#1}$}


\begin{thebibliography}{GTGHS20}
	
	\bibitem[AAB{\etalchar{+}}24]{eigenmaps}
	Bernard Akwei, Bobita Atkins, Rachel Bailey, Ashka Dalal, Natalie Dinin,
	Jonathan Kerby-White, Tess McGuinness, Tonya Patricks, Luke Rogers, Genevieve
	Romanelli, Yiheng Su, and Alexander Teplyaev.
	\newblock Convergence, optimization and stability of singular eigenmaps.
	\newblock {\em arXiv:2406.19510}, 2024.
	
	\bibitem[AV25]{armstrong2025optimal}
	Scott Armstrong and Raghavendra Venkatraman.
	\newblock Optimal convergence rates for the spectrum of the graph {L}aplacian
	on {P}oisson point clouds.
	\newblock {\em Foundations of Computational Mathematics}, pages 1--26, 2025.
	
	\bibitem[BIK15]{burago2015graph}
	Dmitri Burago, Sergei Ivanov, and Yaroslav Kurylev.
	\newblock A graph discretization of the {L}aplace--{B}eltrami operator.
	\newblock {\em Journal of Spectral Theory}, 4(4):675--714, 2015.
	
	\bibitem[BIK19]{burago2019}
	Dmitri Burago, Sergei Ivanov, and Yaroslav Kurylev.
	\newblock Spectral stability of metric-measure {L}aplacians.
	\newblock {\em Israel J. Math.}, 232(1):125--158, 2019.
	
	\bibitem[Bil17]{billingsley2017probability}
	Patrick Billingsley.
	\newblock {\em Probability and measure}.
	\newblock John Wiley \& Sons, 2017.
	
	\bibitem[BN03]{belkin2003laplacian}
	Mikhail Belkin and Partha Niyogi.
	\newblock Laplacian eigenmaps for dimensionality reduction and data
	representation.
	\newblock {\em Neural computation}, 15(6):1373--1396, 2003.
	
	\bibitem[BN08]{belkin2008towards}
	Mikhail Belkin and Partha Niyogi.
	\newblock Towards a theoretical foundation for {L}aplacian-based manifold
	methods.
	\newblock {\em Journal of Computer and System Sciences}, 74(8):1289--1308,
	2008.
	
	\bibitem[CL06]{coifman2006diffusion}
	Ronald~R Coifman and St{\'e}phane Lafon.
	\newblock Diffusion maps.
	\newblock {\em Applied and computational harmonic analysis}, 21(1):5--30, 2006.
	
	\bibitem[Eva22]{evans2022partial}
	Lawrence~C Evans.
	\newblock {\em Partial differential equations}, volume~19.
	\newblock American Mathematical Society, 2022.
	
	\bibitem[Fol05]{folland2005higher}
	Gerald~B Folland.
	\newblock Higher-order derivatives and {T}aylor’s formula in several
	variables.
	\newblock {\em Preprint}, pages 1--4, 2005.
	
	\bibitem[GBT21]{green2021minimax}
	Alden Green, Sivaraman Balakrishnan, and Ryan~J Tibshirani.
	\newblock Minimax optimal regression over {S}obolev spaces via {L}aplacian
	eigenmaps on neighborhood graphs.
	\newblock {\em Preprint arXiv:2111.07394}, 2021.
	
	\bibitem[GK06]{gine2006empirical}
	Evarist Gin{\'e} and Vladimir Koltchinskii.
	\newblock Empirical graph {L}aplacian approximation of {L}aplace-{B}eltrami
	operators: large sample results.
	\newblock {\em Lecture Notes-Monograph Series}, pages 238--259, 2006.
	
	\bibitem[GTGHS20]{garcia2020error}
	Nicol{\'a}s Garc{\'\i}a~Trillos, Moritz Gerlach, Matthias Hein, and Dejan
	Slep{\v{c}}ev.
	\newblock Error estimates for spectral convergence of the graph laplacian on
	random geometric graphs toward the {L}aplace--{B}eltrami operator.
	\newblock {\em Foundations of Computational Mathematics}, 20(4):827--887, 2020.
	
	\bibitem[JMS08]{jones2008manifold}
	Peter~W Jones, Mauro Maggioni, and Raanan Schul.
	\newblock Manifold parametrizations by eigenfunctions of the {L}aplacian and
	heat kernels.
	\newblock {\em Proceedings of the National Academy of Sciences},
	105(6):1803--1808, 2008.
	
	\bibitem[KG00]{koltchinskii2000random}
	Vladimir Koltchinskii and Evarist Gin{\'e}.
	\newblock Random matrix approximation of spectra of integral operators.
	\newblock {\em Bernoulli}, pages 113--167, 2000.
	
	\bibitem[Kob21]{Kobayashi}
	Shoshichi Kobayashi.
	\newblock {\em Differential geometry of curves and surfaces}.
	\newblock Springer Undergraduate Mathematics Series. Springer, Singapore,
	[2021] \copyright 2021.
	\newblock Translated from the revised 1995 Japanese edition by Eriko Shinozaki
	Nagumo and Makiko Sumi Tanaka.
	
	\bibitem[LS22]{lin2022varadhan}
	Chen-Yun Lin and Christina Sormani.
	\newblock From {V}aradhan's limit to eigenmaps: A guide to the geometric
	analysis behind manifold learning.
	\newblock {\em Preprint arXiv:2210.10405}, 2022.
	
	\bibitem[Sin06]{singer2006graph}
	Amit Singer.
	\newblock From graph to manifold {L}aplacian: The convergence rate.
	\newblock {\em Applied and Computational Harmonic Analysis}, 21(1):128--134,
	2006.
	
	\bibitem[VS23]{Venkatraman2023}
	Raghavendra Venkatraman and Dejan Slep{\v{c}}ev.
	\newblock Discrete to continuum limits in graph-{L}aplacian-based fractional
	{S}obolev spaces.
	\newblock {\em In preparation}, 2023.
	
\end{thebibliography}
\end{document}